\documentclass{amsart}
\usepackage{math}
\usepackage{lineno,tikz,mathabx,mathrsfs,mathtools,amsrefs,hyperref}
\usetikzlibrary{cd,quotes,angles,decorations,arrows,calc,automata}

\newenvironment{fsa}[1][auto]{\begin{tikzpicture}[->,>=stealth',
    shorten >=1pt,auto,node distance=3cm,double distance between line centers=0.45ex,
    initial text=,accepting/.style=accepting by arrow,
    every state/.style={inner sep=3pt,minimum size=0pt},
    every loop/.style={looseness=12},semithick,#1]}{\end{tikzpicture}}

\newtheorem{mainthm}{Theorem}

\newtheorem{maincor}[mainthm]{Corollary}

\def\mathvisiblespace{\text{\textvisiblespace}}

\begin{document}

\title{Automatic actions I. Bounded automata and orbits}
\author{Laurent Bartholdi}
\address{Institut Camille Jordan, Université Claude Bernard, Lyon}
\email{laurent.bartholdi@gmail.com}
\date{September 8th, 2025}
\thanks{The author gratefully acknowledges partial support from the SNSF grant TMAG-2\_216487/1}
\begin{abstract}
  We develop the theory of \emph{automatic actions}: (semi)groups acting by $\omega$-regular transformations on an $\omega$-regular language, showing that it covers a large class of heretofore-unrelated examples.

  We focus on the subclass of actions by \emph{bounded} $\omega$-regular transformations, those for which the B\"uchi automata encoding the action do not have two connected non-trivial cycles.

  We show that, for bounded actions of inverse semigroups, the orbit relation is also $\omega$-regular.

  We deduce a number of corollaries, in particular decidability, for such actions, of minimality, topological transitivity, aperiodicity, and order of elements. More generally, every first-order statement over the space of the action, involving the action of specific semigroup elements as well as the relation ``being in the same orbit'', is decidable.

  We also apply this result to the study of Julia sets of post-critically finite polynomials, and show that the encoding of Fatou components is also computable; thus every first-order statement involving intersection, disjointness etc. of Fatou components or their full orbit under the polynomial, is decidable.
\end{abstract}
\maketitle

\section{Introduction}
A variety of group actions $G\curvearrowright X$ share a common setup: an encoding of $X$ by an \emph{$\omega$-regular language $L$}, such that for all $g\in G$ the graph $R_g\coloneqq\{(x,x\cdot g)\}\mid x\in X\}$ of its action is encoded by an $\omega$-regular language $\subseteq L\times L$. Recall that an $\omega$-regular language $L$ is a collection of right-infinite words over a finite alphabet $\Sigma$, that it is recognized by a \emph{B\"uchi automaton}: a finite directed graph with edge labels in $\Sigma$ and some vertices designated as \emph{initial} or \emph{residual}. The language $L$ consists in all $\omega$-words that can be read along a right-infinite path starting at an initial vertex and passing infinitely often through a residual vertex. Likewise, $R_g$ is recognized by a B\"uchi automaton over the alphabet $\Sigma\times\Sigma$, sometimes called a \emph{B\"uchi transducer}. By an \emph{encoding} we mean a map $\pi\colon L\twoheadrightarrow X$ whose kernel $\{(x,y)\mid \pi(x)=\pi(y)\}$ is $\omega$-regular, again with alphabet $\Sigma\times\Sigma$. See~\S\ref{ss:automata} for more complete definitions and illustrations; suffice it here to say that all boolean and comparison operations on $\omega$-regular languages are effective.

Main examples of such \emph{automatic actions} are: automatic groups in the sense of Cannon-Epstein-Holt-Levy-Patterson-Thurston~\cite{epstein-:wp}, and in particular word-hyperbolic groups; automata groups in the sense of Glushkov~\cite{glushkov:ata}, and in particular the constructions by Aleshin and Grigorchuk, and Nekrashevych's iterated monodromy groups~\cite{nekrashevych:ssg}; and the $\Z$-action on substitutional subshifts.

In the language of logic, these are actions admitting an \emph{automatic presentation} in the sense of Khoussainov-Nerode~\cite{khoussainov-nerode:aps}, with $\omega$-regular languages as in Blumensath-Gr\"adel~\cite{blumensath-graedel:as} rather than the more common regular ones. Indeed, an automatic action is nothing but a presentation of the structure $\langle X,\{R_g\}_{g\in G}\rangle$.

By a fundamental result, due to Kuske-Lohrey~\cite{kuske-lohrey:counting} and B\'ar\'any-Kaiser-Rabinovich~\cite{barany-kaiser-rabinovich:cardinality}, the first-order theory of an automatically-presented structure $\mathscr X=\langle X,R_1,\dots\rangle$ is decidable. Namely, there is an algorithm that, given a first-order statement (a logical expression with quatification $\forall,\exists$ over elements of $X$, boolean connectives, and relations $R_i$), decides whether it holds or not in $\mathscr X$. Furthermore, these statements hold even if first-order theory is enriched with cardinality quantifiers $\exists^\kappa$ for $\kappa\in\N\cup\{\aleph_0,2^{\aleph_0}\}$ a cardinal (those are the only ones that may occur), and congruence quantifiers $\exists^{k\bmod n}$, requiring the existence of precisely that many solutions.

This article focusses on a subclass of automatic actions, the \emph{bounded actions}, see~\S\ref{ss:bounded}. Consider an automaton $\mathscr A_g$ recognizing the relation $R_g\subseteq L\times L$, and declare a vertex $v$ of $\mathscr A_g$ \emph{trivial} if all right-infinite paths starting at $v$ carry a label of the form $(w,w)$; equivalently, if making $v$ the initial state of $\mathscr A_g$ will make it recognize a subset of the diagonal in $L\times L$. Then $\mathscr A_g$ is called \emph{bounded} if there are only finitely many right-infinite paths in it that do not reach a trivial state, and $G\curvearrowright X$ is a bounded action if it admits an encoding such that all $R_g$ may be given by bounded automata.

Even though this condition may seem very restrictive, the class of bounded actions contains most known examples of automata groups, all iterated monodromy groups of post-critically finite polynomials (namely, polynomials $\in\C[z]$ all of whose critical points have finite orbits, see~\S\ref{ss:cx}), and substitutional subshifts, see~\S\ref{ss:morphic}.

A number of examples will be given in~\S\ref{ss:automata}. The prototype of a bounded, automatic $\Z$-action is the \emph{odometer}, with $X=\Z_2$ the $2$-adics, represented by their digit expansions as $L=\{0,1\}^\omega$, and the action of $\Z$ on $X$ by addition. The generator $+1$ acts by ``addition with carry'': while it sees a $1$, it prints a $0$ and carries to the next position, till it sees a $0$ and prints a $1$. There is therefore a unique infinite path in the automaton that does not reach a trivial state:
\[
  \begin{fsa}[baseline=0]
  \node[state,initial above,accepting below] (a) at (0,0) {$1$};
  \node[state,accepting below] (e) at (3,0) {$0$};
  \path (a) edge[loop left] node[left] {$(1,0)$} () edge node {$(0,1)$} (e)
  (e) edge[loop right] node[below=1mm] {$(1,1)$} () edge[loop above] node[right] {$(0,0)$} ();
\end{fsa}
\]
Substitutional subshifts are encoded by their desubstutition sequence, and the shift generator acts also as a kind of ``addition with carry'', but in a (typically) non-integer base. The ``$+1$'' automaton has an infinite path corresponding to an infinite sequence of carries, which occurs only when encoding the point at the boundary of arbitrarily large blocks of the substitution.

Furthermore, it costs nothing, and leads to many more applications, to consider actions by inverse semigroups (namely, actions by partial bijections) rather than by groups. The main result of this article is, in simplified form:
\begin{mainthm}[see Theorem~\ref{thm:orbits}]
  Let $G\curvearrowright X$ be a bounded automatic action by homeomorphisms on a compact $X$, encoded by a language $L$. Then the \emph{orbit relation} $\mathscr O_G=\{(x,x\cdot g)\mid x\in X,g\in G\}=\bigcup_{g\in G}R_G$ effectively admits an encoding by an $\omega$-regular language $\subseteq L\times L$.
\end{mainthm}

It follows that, for bounded actions, $\mathscr O_G$ can be considered as part of the structure, and therefore used in first-order statements without impairing their decidability. Thus
\begin{maincor}
  It is decidable, given a bounded automatic action for which the encoding $L\twoheadrightarrow X$ is continuous, whether it is
  \begin{itemize}
  \item minimal (every orbit is dense);
  \item topologically transitive (there is a dense orbit);
  \end{itemize}
\end{maincor}
\begin{proof}
  The topological closure operations are also effective (and are considered in more generality in the main text), so the relation $\mathscr R_G=\{(x,y)\mid y\in\overline{x\cdot G}\}$ also admits an $\omega$-regular encoding. Thus the first statement is expressed logically as
  \[\forall x:\forall y:(x,y)\in\mathscr R_G,\]
  the second as
  \[\exists x:\forall y:(x,y)\in\mathscr R_G.\qedhere\]
\end{proof}

The following generalizes a result of Bondarenko-W\"achter~\cite{bondarenko-wachter:orbits}:
\begin{maincor}[= Corollaries~\ref{cor:infiniteorbit} and~\ref{cor:infinite}]
  It is decidable, given a bounded automatic action $G\curvearrowright X$, whether
  \begin{itemize}
  \item the action is aperiodic (there is no finite orbit);
  \item $G$ has an infinite orbit;
  \item $G$ is infinite.
  \end{itemize}
\end{maincor}
\begin{proof}
  The first statement is
  \[\forall x:\exists^{\ge\omega}y:(x,y)\in\mathscr O_G,\]
  and the second is
  \[\exists x:\exists^{\ge\omega}y:(x,y)\in\mathscr O_G.\]
  For the third one, D'Angeli-Francoeur-Rodaro-W\"achter prove in~\cite{dangeli-francoeur-rodaro-wachter:orbits} that an automata group is infinite if and only if it has an infinite orbit; their argument applies here.
\end{proof}
In particular, it is decidable given a bijection $g\colon X\righttoleftarrow$ given by a bounded automaton, whether $g$ has infinite order. This extends a result of Bondarenko-Sidki-Zapata for automata~\cite{bondarenko-b-s-z:conjugacy}.

Note that, if the bijection $g$ were merely given by an automaton, then the order of $g$ could not be determined in general; see~\cite{bartholdi:wordorder}. There are examples of automatic actions for which the orbit relation is not $\omega$-regular, see~\S\ref{ss:bernoulli}.

We finally consider some applications on iterated monodromy groups; see~\S\ref{ss:cx}. Nekrashevych develops a duality theory between expanding dynamical systems and certain group actions, which yield an encoding of the Julia set of a rational map $f$ by \emph{left-}infinite words, under an $\omega$-regular equivalence relation. Orbits of certain subgroups are related to certain dynamically-significant subsets of the Julia set; in particular, the subgroup generated by a loop around a Fatou component is related to that Fatou component. We deduce:
\begin{cor}[= Corollary~\ref{cor:fatou}]
  Let $f$ be a post-critically finite polynomial. Then for every periodic critical point $z$, the relation ``$x,y$ are on the boundary of a preimage of the Fatou component of $z$'' is effectively $\omega$-regular. For every Fatou component $F$, the relation ``$x,y$ are on the boundary of $F$'' is also effectively $\omega$-regular.
\end{cor}

It follows that every first-order statement over the Julia set of $f$, involving relations of the type ``$x$ is on the boundary of a certain Fatou component'', or ``$x,y$ are on the boundary of a preimage of a certain Fatou component'', is decidable.

\subsection{Reddite ergo quae sunt Caesaris, Caesari}
The article~\cite{bondarenko-wachter:orbits} was a great source of inspiration for this article. Their proof makes implicit use of an $\omega$-regular language of orbits. They consider the restrictive setting of self-similar groups on trees; while this article drops the requirements of self-similary, actions on trees, actions on full shifts, and even group actions.

Indeed the main text is written in the language of inverse semigroups, which are the natural setting for action on languages $\neq\Sigma^\omega$.

The encoding of substitutional subshifts is essentially covered by the theory of Bratteli-Vershik diagrams; a novelty of the present approach is that it covers all expanding substitutions, and not only the primitive ones. This opens the door to studying, via B\"uchi automata, fundamental questions such as conjugacy of subshifts; this has, for now, only been explored in the setting of primitive substitutions, or equivalently minimal subshifts.

We note finally that the automatic structures we consider only encode the \emph{action}, namely the set $X$ on which $G$ acts, along with relations on $X$ given by elements of $G$. There is no encoding of the group $G$ itself, so that quantifications ``$\exists g:\cdots$'' are not at our disposal. It is possible, in some favourable cases of substitutional subshifts, to also allow quantifications over the group $G=\Z$ and nevertheless have a decidable first-order theory; this will be the topic of a further article~\cite{bartholdi-m:actions2}.

I am deeply grateful to Ruiwen Dong and Ivan Mitrofanov for valuable discussions and feedback on this text.

\section{Automata and transducers}\label{ss:automata}
We start by recalling basic notions on $\omega$-automata. A \emph{B\"uchi automaton} on an alphabet $\Sigma$ is a finite directed graph $\mathscr A$, with edge labels in $\Sigma$, and whose vertices $s\in S$ are called \emph{states}; a set of \emph{initial} vertices $I\subseteq S$; and a \emph{residual} subset of vertices $F\subseteq S$. The accepted \emph{language} $L(\mathscr A)\subset \Sigma^\omega$ consists of those labels of right-infinite paths in $\mathscr A$ that start at an initial vertex and visit infinitely often a residual state. In formulas,
\[L(\mathscr A) = \Big\{x\in\Sigma^\omega\mid \exists\text{path }\pi\colon\omega\to S\text{ labeled }x\text{ with }\pi(0)\in I,\#\pi^{-1}(F)=\infty\Big\}.\]
A language $L\subseteq\Sigma^\omega$ is \emph{$\omega$-regular} if is it the language of a B\"uchi automaton.

Note that $L$ has the topology of a subspace of the Cantor set $\Sigma^\omega$, so it is totally disconnected. It need \emph{not} be a closed subspace --- if $L$ is closed, then it can be expressed as $L(\mathscr A)$ for an automaton in which all states are residual. Such automata are called \emph{safety automata}. These are closely related to classical finite state automata: they admit a unique minimal representation, and their language is the closure of a language of finite words.

\begin{thm}[B\"uchi~\cites{buchi:automata,buchi:mso}]
  The class of $\omega$-regular languages over $\Sigma$ is effectively closed under Boolean operations ($\cup,\cap,\neg$) and images and preimages under homomorphisms.

  It is decidable, given an $\omega$-automaton $\mathscr A$, whether $L(\mathscr A)$ is empty.
\end{thm}

It is also easy to see that the topological closure of $L$ is again automatic: an automaton for $\overline L$ may be defined by first removing useless states of $\mathscr A$ (those that cannot reach a cycle containing a residual state), and then making all states residual.

Let $L$ be an $\omega$-regular language. A binary relation $R\subseteq L\times L$ is \emph{$\omega$-regular} if it is the language of an $\omega$-automaton with alphabet $\Sigma\times\Sigma$.

We recall the usual algebraic operations on relations: $R^{-1}=\{(y,x)\mid(x,y)\in R\}$ and $RS = \{(x,z)\mid\exists y:(x,y)\in R,(y,z)\in S\}$, the \emph{total} relation $E=L\times L$ and \emph{diagonal} relation $\Delta=\{(x,x)\mid x\in L\}$, and the usual properties: \emph{univalent} relations ($\forall x:\exists^{\le1}y:R(x,y)$, equivalently $R R^{-1}\subseteq\Delta$), \emph{total} relations ($\forall x:\exists^{\ge1}y:R(x,y)$, equivalently $R R^{-1}\supseteq\Delta$, equivalently $R E=E$), \emph{injective} relations ($R^{-1}R\subseteq\Delta$), \emph{surjective} relations ($R^{-1}R\supseteq\Delta$), \emph{invertible} relations ($R^{-1}R=R R^{-1}=\Delta$), \emph{equivalence} relations ($R^2=R=R^{-1}\supseteq\Delta$), etc., and ponder whether they are $\omega$-regular.

It is easy to see that the composition of two $\omega$-regular relations, and the inverse of an $\omega$-regular relation, are $\omega$-regular. This merely follows from the closure of $\omega$-regular relations under intersection and projection. In particular, any collection of $\omega$-regular relations generates a semigroup called \emph{$\omega$-automatic semigroup}, which is a group if these relations are invertible, and is an inverse semigroup if these relations are univalent and injective.
\begin{defn}\label{def:aa}
  Let $G$ be a monoid acting on a set $X$. The action is called \emph{automatic} if there is an $\omega$-regular language $L$ and a surjective map $\pi\colon L\twoheadrightarrow X$ such that
  \[\mathscr G_g\coloneqq\{(x,y)\in L\mid \pi(x)=\pi(y)\cdot g\}\text{ is $\omega$-regular for all }g\in G.\]
  The language $L$ and the sets $\mathscr G_g$ define an \emph{automatic presentation} of the action.

  If furthermore $\pi$ is a bijection, then this presentation is called \emph{injective}.
\end{defn}
By the above remarks, if $G$ is generated by a set $S$ then it suffices to verify that the graph of $s$ is $\omega$-regular for all $s\in S$. Note that, from the definition, the kernel $\{(x,y)\mid \pi(x)=\pi(y)\}$ of $\pi$ is an $\omega$-regular relation.

\subsection{Automatic structures}\label{ss:automatic}
This notion fits nicely in the framework of \emph{$\omega$-automatic presentations of structures}, see~\cite{blumensath-graedel:as}. Consider a structure $\mathfrak X=\langle X,R_1,\dots\rangle$ consisting of a set $X$ and relations $R_i$, each with arity $n_i$ so $R_i\subseteq X^{n_i}$. An \emph{$\omega$-automatic presentation} of $\langle X,R_1,\dots\rangle$ consists in an $\omega$-regular language $L\subseteq\Sigma^\omega$, an $\omega$-automatic equivalence relation ${\approx}\subseteq L\times L$, and $\omega$-automatic relations $r_i\subseteq L^{n_i}$ such that there exists a bijection $\pi\colon L/{\approx}\leftrightarrow X$ with $r_i(x_1,\dots,x_{n_i})\Leftrightarrow R_i(\pi(x_1),\dots,\pi(x_{n_i}))$ for all $i$ and all $x_1,\dots,x_{n_i}\in L$. The presentation is \emph{injective} if $\approx$ is the identity relation.

We may thus formulate our definition as follows: an \emph{automatic action} of $G$ on $X$ is an $\omega$-automatic presentation of $\langle L,(\cdot g)_{g\in G}\rangle$.

Beware that this is much weaker than requiring the action itself to be automatic, in the sense that there would be a coding of $G$ by an $\omega$-regular language $M$ and an $\omega$-regular relation $\subseteq M\times L\times L$ expressing the action; see~\cite{bartholdi-m:actions2} for a discussion of when this happens.

The following fundamental result (proven by Kuske and Lohrey in the case of injective automatic structures (i.e.\ the coding map $L\twoheadrightarrow X$ is injective), and by B\'ar\'any, Kaiser and Rabinovich in the general case) gives ``automatic'' proofs of all first-order statements in an $\omega$-automatically presentable structure:
\begin{thm}[\cite{kuske-lohrey:counting,barany-kaiser-rabinovich:cardinality}]
  Let $\mathscr A$ be an $\omega$-automatic structure. Then there is an algorithm that, given any sentence in the first order logic of $\mathscr A$, determines whether it is true. This holds even when the first-order logic is enriched with counting and congruence quantifiers: $\exists^\kappa$ to mean there are precisely $\kappa$ solutions to an existential statement ($\kappa\le\aleph_0$ or $\kappa=2^{\aleph_0}$), and $\exists^{a\pmod m}$ to mean that the number of solutions is $\equiv a\pmod m$.\qed
\end{thm}
In fact, and we shall need this extra information, \emph{every predicate (relation) defined by first-order quantifiers and logical operations is effectively automatic.}

\subsection{Topology}\label{ss:top}
Let us consider now the case that $X$ is a topological space, and that the action of $G$ is by homeomorphisms. Note that we have not required the surjective map $\pi\colon L\twoheadrightarrow X$ to be continuous for the topology induced on $L$ as a subset of $\Sigma^\omega$. If this is the case, we call the action \emph{topologically automatic}.

We can in fact afford a little bit more flexibility: let $(\mathcal U_i)_{i\in I}$ be a basis of the topology of $X$; extend the encoding of $X$ to an encoding of $X\sqcup I$, and add the relation $\{(x,i)\mid\pi(x)\in\mathcal U_i\}$ to the structure.

If $X$ is metrizable and compact, we can encode the uniform structure on $X$ in a simple manner, as follows: extend the encoding of $X$ to an encoding of $X\sqcup\N$, by extending the language $L$ to $L\sqcup\{1^*\mathvisiblespace^\omega\}$, with $\pi(1^n\mathvisiblespace^\omega)=n\in\N$; we represent $X$ as before using $L$, and represent $\N$ in unary. Let $(\mathcal U_n)_{n\in\N}$ be a family of neighbourhoods of the diagonal $\Delta(X)$ in $X\times X$. It may be represented the relation
\[\mathscr U\coloneqq\{(x,y,1^n\mathvisiblespace^\omega)\in L\times L\times(1^*\mathvisiblespace^\omega)\mid (\pi(x),\pi(y))\in \mathcal U_n\}.\]
If there exists a family of neighbourhoods $(\mathcal U_n)_{n\in\N}$ as above such that $\mathscr U$ is $\omega$-regular, we shall call the action \emph{uniformly automatic}.

In particular, if the map $\pi\colon L\to X$ is continuous for the topology of $L$ as a subset of $\Sigma^\omega$, then $\mathcal U_n=\{((\pi(x),\pi(y))\mid x,y\text{ agree on their first $n$ letters at least}\}$ generates a uniform structure on $X$. Assuming first that $\pi\colon L\to X$ is a homeomorphism, $\mathscr U$ is the relation
\[\{(x,y,1^n\mathvisiblespace^\omega)\mid x\text{ and $y$ agree on their first $n$ letters at least}\},\]
given by the automaton
\[\begin{fsa}
    \node[state,initial above] (a) at (0,0) {};
    \node[state,accepting below] (b) at (3,0) {};
    \draw (a) edge[loop left,looseness=25] node {$\Delta(\Sigma)\times\{1\}$} (a)
    (a) edge node {$\Sigma\times\Sigma\times\{\mathvisiblespace\}$} (b)
    (b) edge[loop right,looseness=25] node {$\Sigma\times\Sigma\times\{\mathvisiblespace\}$} (b);
  \end{fsa},\]
where here and below we write $\Delta(x)=(x,x)$ for the diagonal map $Y\to Y\times Y$, for any $Y$. Indeed accepting paths in the automaton consist in finitely many loops at the leftmost state followed by infinitely many edges labeled $\Sigma\times\Sigma\times\{\mathvisiblespace\}$, so their label is of the form $(x_0,x_0,1)\dots(x_{n-1},x_{n-1},1)(x_n,y_n,\mathvisiblespace)(x_{n+1},y_{n+1},\mathvisiblespace)\dots$. In the general case that $\pi\colon L\to X$ is not injective, this relation $\mathscr U$ must be saturated with the kernel of $\pi$, namely replaced by
\[\{(x,y,n)\in L\times L\times(1^*\mathvisiblespace^\omega)\mid\exists x'\approx x,y'\approx y,(x',y',n)\in\mathscr U\}=({\approx}\times1)\cdot\mathscr U\cdot({\approx}\times1).\]

\subsection{Automatic groups}
A special case of automatic actions has been considered since the 1980's, under the name of \emph{automatic groups}, see~\cite{epstein-:wp}. Consider a group $G$ with a finite generating set $S$. The group $G$ is called \emph{automatic} if, firstly, there is a \emph{regular normal form}: a regular language $N\subseteq S^*$ for which the evaluation map $\pi\colon N\to G$ is surjective, and secondly, for all $s\in S\cup\{1\}$ there is a \emph{multiplication automaton}: a finite state automaton $\mathscr A_s$ with alphabet $(S\sqcup\{\mathvisiblespace\})^2$ accepting the language
\begin{equation}\label{eq:agautomata}
  \left\{(u_1,v_1)\cdots(u_n,v_n)\left| \begin{array}{c}
                                     \max(\ell,m)=n,\\
                                     u_1\dots u_\ell\in N, u_{\ell+1}=\dots=\mathvisiblespace,\\
                                     v_1\dots v_m\in N,v_{m+1}=\dots=\mathvisiblespace,\\
                                     \pi(u_1\dots u_\ell)=\pi(v_1\dots v_m)s
                                   \end{array}\right.\right\}.
\end{equation}
If we assume (as we may) that the evaluation map $\pi\colon N\to G$ is bijective, then the multiplication automaton $\mathscr A_s$ converts the normal form of any element $g\in G$ to the normal form of $g s$. Automatic groups cover important classes of groups, in particular all Abelian groups, word-hyperbolic groups, and fundamental groups of $3$-manifolds without nilpotent or soluble pieces; and possess valuable properties such as being finitely presented and having word problem solvable in quadratic time.

It follows readily from the definition that the action of an automatic group on itself is $\omega$-regular in the sense of Definition~\ref{def:aa}. Indeed take as $\omega$-regular language the set of normal forms with an infinite amount of padding: $L=\{w\mathvisiblespace^\omega\mid w\in N\}$. Then the paddings of the languages in~\eqref{eq:agautomata} are $\omega$-regular relations. We have proven:
\begin{prop}
  An automatic group $G$ is precisely the same thing as an $\omega$-regular action of $G$ on $X=G$ by right multiplication such that the language $L\subseteq S^*\mathvisiblespace^\omega\subseteq(S\sqcup\mathvisiblespace)^\omega$ has as alphabet a generating set of $G$, the map $\pi\colon L\to X$ being given by evaluation (with $\mathvisiblespace$ evaluating to $1$).\qed
\end{prop}

A large part of the theory of automatic groups survives if one drops the assumption that the map $\pi$ be given by evaluation. One arrives then at the notion of \emph{Cayley-automatic groups}, see~\cite{kharlampovich-k-m:automatic}; the terminology comes from the Cayley graph of $G$ being given by an automatic structure. We record:
\begin{prop}
  A \emph{Cayley-automatic group} is precisely the same thing as an $\omega$-regular action of a finitely generated group $G$ on $X=G$ by right multiplication.
\end{prop}
\begin{proof}
  The only subtle point to note Cayley-automatic groups are defined in terms of automatic structures, not $\omega$-automatic structures. However, a countable $\omega$-automatic structure is also automatic, see~\cite{kaiser-rubin-barany:cardinality}.
\end{proof}
 
\subsection{Automata groups}\label{ss:autgroup}
Let $\mathscr M$ be a \emph{Mealy machine}: a finite graph with vertex set $S$ called its \emph{set of states}, and on every edge, an \emph{input} and \emph{output} symbol in $\Sigma$. Assume that, at every $s\in S$ and for every $\sigma\in\Sigma$, there is a unique edge issued from $s$ and with input symbol $\sigma$. Then every choice of initial state $s\in S$ induces a transformation $T_s\colon\Sigma^*\righttoleftarrow$, by the rule: given $\sigma_1\dots\sigma_n\in\Sigma^*$, there is a unique path starting at $v_0$ and with input label $\sigma_1\dots\sigma_n$; the output label along the same path is, by definition, $T_s(\sigma_1\dots\sigma_n)$.

This action can be viewed geometrically: $\Sigma^*$ may be viewed as the vertex set of a tree $\mathcal T_\Sigma$, rooted at the empty word $\varepsilon$, and with an edge between $\sigma_1\dots\sigma_n$ and $\sigma_1\dots\sigma_n\sigma_{n+1}$ for all $\sigma_i\in\Sigma$. Then the transformation $T_s$ is an isometry of $\mathcal T_\Sigma$.

There are various definitions of automata groups; one of them is that it is a group $G=\langle T_{s_1},\dots,T_{s_n}\rangle$ for some choices of initial state $s_1,\dots,s_n$ in a Mealy machine $\mathscr M$. (One could choose each $s_i$ initial in its own Mealy machine $\mathscr M_i$, but this is not more general, since one may let $\mathscr M$ be the disjoint union of all $\mathscr M_i$.) Another, more restrictive definition is that it is a group of the form $G(\mathscr M)=\langle T_s:s\in S\rangle$ where $s$ ranges over \emph{all} possible choices of initial state in a Mealy machine $\mathscr M$. We call these \emph{self-similar automata groups} to distinguish them from the general class.

The definitions above make sense for isometries and groups of isometries of $\mathcal T_\Sigma$. Every isometry $g\colon\mathcal T_\Sigma\righttoleftarrow$ has \emph{states}, given by the tail actions of $g$: for $v\in\Sigma^*$, consider the composition of maps
\[\Sigma^*\to v\Sigma^*\to (v\Sigma^*)\cdot g=(v\cdot g)\Sigma^*\to\Sigma^*,\]
the first map being pre-catenation of $v$, the second being $g$, and the third being removal of the initial $v\cdot g$. This composition is again a tree isometry, called the \emph{state} $g@v$ of $g$ at $v$. An isometry is called \emph{finite-state} if its set of states is finite, and then $g$ may be described by an initial Mealy machine, with stateset $\{g@v\mid v\in\Sigma^*\}$, initial state $g=g@\varepsilon$, and an edge from $g@v$ to $g@(v\sigma)$ with input label $\sigma$ and output label $\sigma\cdot g$. A group is \emph{finite-state} if it consists of finite-state elements; and a group is \emph{self-similar} if it contains the states of all its elements.

Let $G$ be a self-similar group. There is then a naturally associated \emph{wreath decomposition} map, which is an injective homomorphism
\[\Psi\colon\begin{cases} G &\to G\wr\operatorname{Sym}(\Sigma)=G^\Sigma\rtimes\operatorname{Sym}(\Sigma)\\
              g &\mapsto (g@\sigma)_{\sigma\in\Sigma}\cdot(g\restriction\Sigma).
            \end{cases}
\]
If furthermore $G$ is finitely generated and finite-state, then we may assume that it is generated by a finite set $S$ closed under taking states, so $\Psi(S)\subseteq(S^\Sigma)\cdot\operatorname{Sym}(\Sigma)\subseteq(S\times\Sigma)^\Sigma$, and we recover a Mealy machine with stateset $S$ and an edge from $s$ to $t$ with input label $\sigma$ and output label $\tau$ whenever $\Psi(s)(\sigma)=(t,\tau)$. Thus the descriptions by Mealy machines, tree automorphisms and wreath decompositions are readily convertible into each other.

A prominent example of self-similar automata group is the \emph{Grigorchuk group} $G(\mathscr M)$ for the Mealy machine
\begin{equation}\label{eq:grigorchuk}
  \mathscr M=\begin{fsa}[baseline=10mm]
  \node[state,accepting above] (b) at (1.4,3) {$b$};
  \node[state,accepting] (d) at (4.2,3) {$d$};
  \node[state,accepting] (c) at (2.8,0.7) {$c$};
  \node[state] (a) at (0,0) {$a$};
  \node[state,accepting below] (e) at (5.6,0) {$1$};
  \path (b) edge node[left] {$\Delta1$} (c) edge node[left] {$\Delta0$} (a)
        (c) edge node {$\Delta1$} (d) edge node {$\Delta0$} (a)
        (d) edge node[above] {$\Delta1$} (b) edge node {$\Delta0$} (e)
        (a) edge[bend right=30] node[below] {$(0,1)$} node[above] {$(1,0)$} (e)
        (e) edge[loop right] node[above=1mm] {$\Delta0$} node[below=1mm] {$\Delta1$} (e);
\end{fsa}\end{equation}
Thus $G=\langle a,b,c,d\rangle$, for the wreath decomposition
\[\Psi\colon\begin{cases} G &\to G\wr\operatorname{Sym}(\{0,1\})=G^2\rtimes\langle \tau_{01}\rangle\\
              a &\mapsto (1,1)\cdot\tau_{01},\\
              b &\mapsto (a,c),\\
              c &\mapsto (a,d),\\
              d &\mapsto (1,b).
            \end{cases}
\]
          
\begin{thm}[Grigorchuk~\cite{grigorchuk:burnside,grigorchuk:growth}]
  The group $G$ is a finitely generated torsion group of intermediate word growth: the group is infinite yet every element has finite order, and the number of group elements expressible as a product of at most $n$ generators is bounded between $\exp(n^{0.5})$ and $\exp(n^{0.77})$.
\end{thm}
(for the sharpest results on estimates of the growth function see~\cite{bartholdi:upperbd,erschler-zheng:growth}).

We may view automata groups as special cases of automatic actions: more precisely,
\begin{prop}
  Let $G$ be an automata group acting on a tree $\mathcal T_\Sigma$; then the action of $G$ on the boundary $\partial\mathcal T_\Sigma=\Sigma^\omega$ is automatic and equicontinuous.
\end{prop}

It would be nice to have a converse statement. If $G\curvearrowright L$ is automatic and equicontinuous, then there exists a metric on $L$ compatible with the topology such that $G$ acts by isometries; and since $L$ is totally disconnected this metric may be assumed to be ultrametric. The action of $G$ on $L$ induces then an action of $G$ on balls in $L$, which form a rooted tree under inclusion. I expect this tree to possess some sort of regularity, in such a manner that the action is given by automata.

\begin{proof}
  The group $G$ acts by isometries on $\mathcal T_\Sigma$, and this action extends to an action by isometries on $\Sigma^\omega$, for the usual metric $d(\sigma_1\sigma_2\dots,\tau_1\tau_2\dots)=\exp(-\min\{i\mid\sigma_i\neq\tau_i\})$.
\end{proof}

An important property of self-similar groups (namely, groups $G$ endowed with an imbedding $\Psi\colon G\to G\wr\operatorname{Sym}(\Sigma)$) is \emph{contraction}. A self-similar group $G$ is \emph{contracting} if there exists a finite subset $\mathscr N\subseteq G$ with the following property: for all $g\in G$ we have $g@v\in \mathscr N$ for all but finitely many $v\in\Sigma^*$. There is clearly a minimal such $\mathscr N$, which is called the \emph{nucleus} of $G$. If $G$ is finitely generated, this is equivalent to contraction of the word metric on $G$: \emph{there are $n\in\N,\lambda<1,K\in\R$ such that $\|g@v\|\le\lambda\|g\|+K$ for all $g\in G$ and all $v\in\Sigma^n$}.

For example, the Grigorchuk group defined above is contracting: the nucleus is $\{1,a,b,c,d\}$, and the metric contraction occurs with $n=1,\lambda=K=\tfrac12$.

The nucleus of a contracting self-similar group is more than a set: we have $g@v\in \mathscr N$ for all $g\in \mathscr N,v\in\Sigma$, so $\mathscr N$ is the vertex set of an automaton with a transition from $g\in \mathscr N$ to $g@v\in \mathscr N$ with input $v$ and output $v\cdot g$. Furthermore every state in this automaton has an incoming transition.

\subsubsection{Nekrashevych's simple torsion groups of intermediate growth}

A small modification of the Grigorchuk group's automaton produces a torsion group of intermediate growth with a simple subgroup of finite index. See~\cite{nekrashevych:palindromic} for the original construction, and~\cite{bartholdi-n-z:lingrowth} for the construction presented here. The group is $G=\langle a_0,a_1,b,c,d\rangle$, with generators given by corresponding initial states in the automaton
\[\begin{fsa}[every node/.style={sloped}]
  \node[state] (b) at (120:1.5) {$b$};
  \node[state] (d) at (0:1.5) {$d$};
  \node[state] (c) at (240:1.5) {$c$};
  \node[state,label=center:$a_0$,minimum size=5.5mm] (a0) at (-3,2) {};
  \node[state,label=center:$a_1$,minimum size=5.5mm] (a10) at (-3,-2) {};
  \node[state,label=center:$a_1$,minimum size=5.5mm] (a11) at (1,-1.732) {};
  \node[state,label=center:$a_1$,minimum size=5.5mm] (e0) at (3.5,-2) {};
  \node[state,label=center:$a_0$,minimum size=5.5mm] (e1) at (3.5,2) {};
  \path (b) edge node[below] {$\Delta1$} (c) edge node[above,pos=0.6] {$(0_0,0_1)$} (a10) edge node[below,pos=0.4] {$(0_1,0_0)$} (a11) edge node[below] {$\Delta0_1$} (a0)
        (c) edge node[below,near end] {$\Delta1$} (d) edge node[above,pos=0.4] {$(0_0,0_1)$} (a10) edge node[below] {$(0_1,0_0)$} (a11) edge node[below,pos=0.6] {$\Delta0_1$} (a0)
        (d) edge node[above,pos=0.6] {$\Delta1$} (b) edge node {$\Delta0_1$} (e0) edge node {$\Delta0_0$} (e1)
        (a0) edge[-implies,double,bend left=5] node[below] {$(0_0,1),(1,0_0)$} (e1)
        (a10) edge[bend right=5] node[below] {$(0_1,1)$} (e0)
        (a11) edge[bend right=5] node[above] {$(1,0_1)$} (e0)
        (e0) edge[loop right] node[sloped=false] {$\Delta1.$} () edge[-implies,double,bend right=20] node[below] {$\Delta1,\Delta0_0$} (e1)
        (e1) edge[bend right=20] node {$\Delta0_1$} (e0);
      \end{fsa}
\]
It acts on an extension of the language $\{0,1\}^\omega$ in which every sequence of consecutive $0$'s alternates as $\dots0_00_1$, ending with $0_1$ in case the next symbol is a $1$. Thus the map $L\to\{0,1\}^\omega$ is almost everywhere $1:1$, and is $2:1$ on sequences ending in $(0_00_1)^\omega$.

Note that the action of $G$ is not equicontinuous, and is even actually expansive.

\subsection{Iterated monodromy groups}\label{ss:img}
Consider a compact metric space $\mathcal M$ and a covering map $f\colon\mathcal M\righttoleftarrow$. The choice of a basepoint $*\in\mathcal M$ yields a \emph{tree of preimages} $\mathcal T_*=\bigsqcup_{n\ge0}f^{-n}(*)$, and an action by monodromy of $\pi_1(\mathcal M,*)$ on $\mathcal T_*$: moving continuously $*$ along a path $\gamma\colon[0,1]\to\mathcal M$ yields a continuous movement $(\mathcal T_{\gamma(t)})_{t\in[0,1]}$ of $\mathcal T_*$, and if the path returns to its starting point then $\mathcal T_*$ returns to itself under a permutation. The \emph{iterated monodromy group} of $f$ is the image of $\pi_1(\mathcal M,*)$ in the permutation group on $\mathcal T_*$.

If one also chooses a \emph{spider}: a choice, for each $\sigma\in f^{-1}(*)$, of a path $\ell_\sigma$ from $*$ to $\sigma$, then the action of $G=\pi_1(\mathcal M,*)$ on $\mathcal T_*$ can be encoded into an (infinite) automaton: its stateset is $G$, its alphabet is $\Sigma=\{\ell_\sigma\mid\sigma\in f^{-1}(*)\}$, and there is a transition from $g\in G$ to $h\in G$ with input $\ell_\sigma$ and output $\ell_\tau$ precisely when $\ell_\sigma\#\widetilde g\sim h\#\ell_\tau$, where $\sim$ denotes homotopy, $\widetilde g$ is the unique lift of $g$ that starts at $\sigma$ and $\#$ denotes path concatenation. Note that $\tau$ is determined as the endpoint of $\widetilde g$.

In case $f$ is expanding, the iterated monodromy group of $f$ is a contracting self-similar group.

An instructive example is that of $\mathcal M=\R/\Z$ with $f(x)=d x$, a degree-$d$ covering. Choose as basepoint $*=0$ with preimages $\{i/d\mid 0\le i<d\}$, and choose for $\ell_{i/d}$ the segment $[0,i/d]$. Then $\pi_1(\mathcal M,*)=\Z$, identifying $n\in\Z$ with the map $t\mapsto n t$; and the automaton has an edge from $n\in\Z$ to $m\in\Z$ with label $(\ell_{i/d},\ell_{j/d})$ whenever $i/d+n/d-j/d=m$. In particular, it has a loop at $1$ with label $(\ell_{(d-1)/d},\ell_{0/d})$ and $d-1$ edges from $1$ to $0$ with labels $(\ell_{i/d},\ell_{(i+1)/d})$ for $i=0,\dots,d-2$:
\begin{equation}\label{eq:adder}\begin{fsa}[baseline=0]
  \node[state,accepting below] (a) at (0,0) {$1$};
  \node[state,accepting below] (e) at (4,0) {$0$};
  \path (a) edge[loop left] node[left] {$(\ell_{(d-1)/d},\ell_{0/d})$} ()
  edge [bend left=20] node[above] {$(\ell_{0/d},\ell_{1/d})$} (e)
  edge [bend right=20] node[below] {$(\ell_{(d-2)/d},\ell_{(d-1)/d})$} (e)
  (e) edge[out=60,in=30,loop] node[right] {$(\ell_{0/d},\ell_{0/d})$} (e)
  edge[out=-30,in=-60,loop] node[right] {$(\ell_{(d-1)/d},\ell_{(d-1)/d})$} (e);
  \node[anchor=center] at (2,0) {$\vdots$};
  \node[anchor=center] at (5.7,0) {$\vdots$};
\end{fsa}\end{equation}
I have not indicated any initial state in this automaton. With initial state $1$, this automaton describes the action of $1\in\Z$; while with initial state $0$ it describes the action of the identity $0\in\Z$. The loop at $1$ represents the ``carry'' to be performed when adding $1$ to a number whose last digit is $d-1$, so the automaton describes the action of $\Z$ by translation on the ``$d$-adics'' $\lim\Z/d^n\Z$, and equivalently arithmetic in base $d$.

The previous example is that of the map $f(z)=z^d$ on $\mathcal M=\C\setminus\{0\}$. Here is another example: the map $f(z)=z^2-1$, with post-critical set $\{0,-1,\infty\}$ so $f$ acts on $\mathcal M=\C\setminus\{0,-1\}$. Choose $*=(1-\sqrt5)/2$ a fixed point of $f$, and $\ell_*$ the constant path and $\ell_{-*}$ an arc of circle in the upper half plane; then $G=\langle a,b\rangle$ with generators $a,b$ respectively corresponding to a path from $*$ to $-1,0$, encircling the puncture counterclockwise, and returning to $*$. The corresponding automaton is
\begin{equation}\label{eq:basilica}
  \begin{fsa}[baseline=0]
  \node[state] (a) at (-2.7,1.3) {$a$};
  \node[state] (b) at (-2.7,-1.3) {$b$};
  \node[state] (e) at (0,0) {$1$};
  \path (a) edge node[above,sloped] {$(\ell_*,\ell_{-*})$} (e) edge[bend left] node {$(\ell_{-*},\ell_*)$} (b)
  (b) edge node[below,sloped] {$(\ell_*,\ell_*)$} (e) edge[bend left] node[left=-1mm] {$(\ell_{-*},\ell_{-*})$} (a)
  (e) edge[loop right] node[above=1mm] {$(\ell_*,\ell_*)$} node[below=1mm] {$(\ell_{-*},\ell_{-*})$} (e);
\end{fsa}
\end{equation}

The path $\ell_{-*}$ is indicated in dotted line, while the generators $a,b$ are in solid and their lifts in dashed lines:

\centerline{\begin{tikzpicture}[>=stealth']
\node[gray] at (0,0) {\includegraphics[width=12cm]{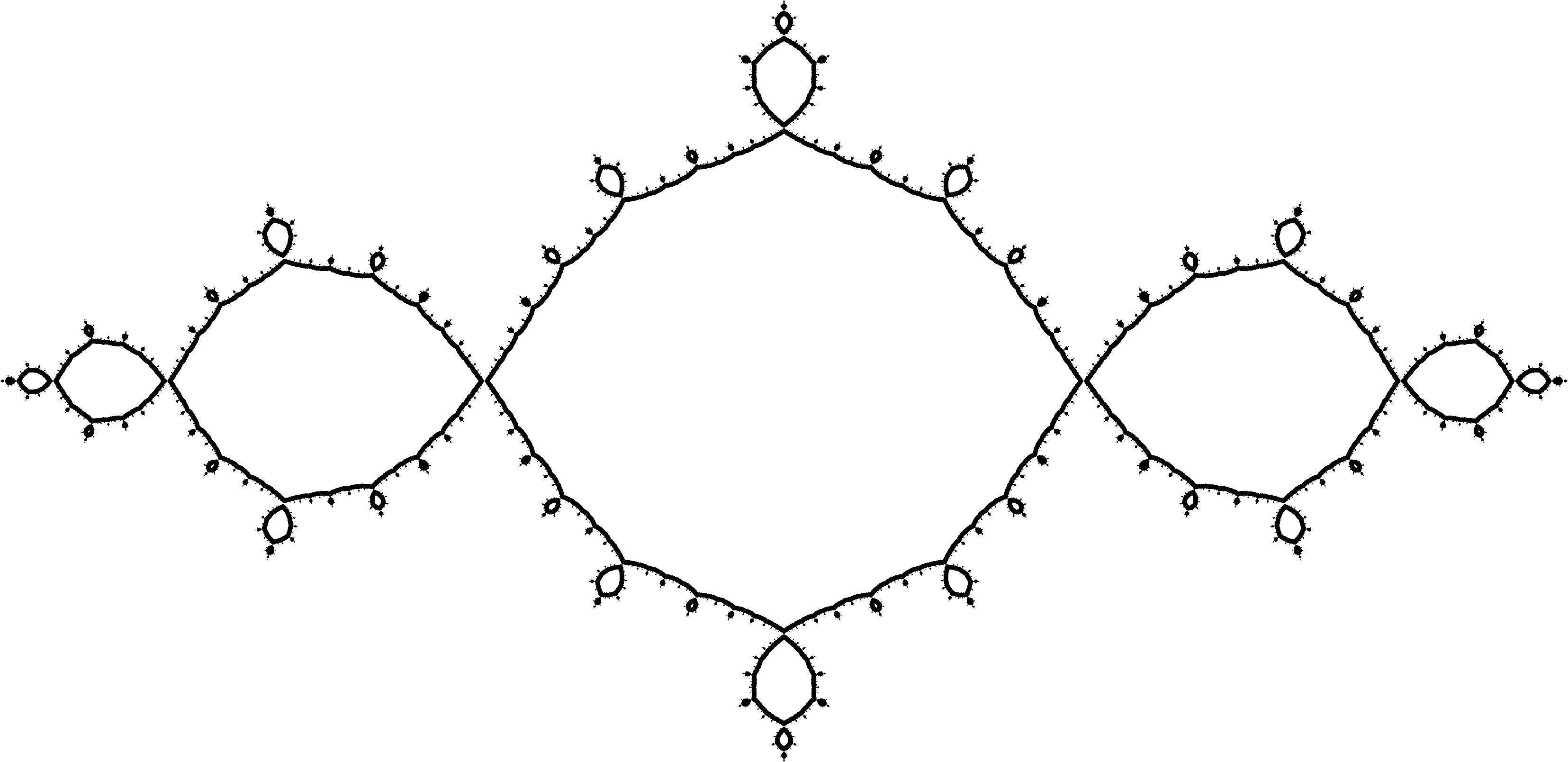}};
\node at (0,0) {$0$};
\node at (-3.6,0) {$-1$};
\node[inner sep=0pt] (x1) at (-2.3,0) {$*$};
\node[inner sep=0pt] (x0) at (2.3,0) {$-*$};
\draw[thick,->] (x1) .. controls (2.5,-4) and (2.5,4) .. node[left] {$a$} (x1);
\draw[thick,->] (x1) .. controls (-5,2.3) and (-5,-2.3) .. node [very near end,below] {$b$} (x1);
\draw[thick,dashed,->] (x1) .. controls (-5.3,2) and (-5.3,-2) .. node [near start,above left=1mm] {$f^{-1}(a)$} (x1);
\draw[thick,dashed,->] (x0) .. controls (5.3,-2) and (5.3,2) .. node [near start,below right=1mm] {$f^{-1}(a)$} (x0);
\draw[thick,dashed,->] (x0) .. controls (1.4,2) and (-1.4,2) .. node [near start,above right=1mm] {$f^{-1}(b)$} (x1);
\draw[thick,dashed,->] (x1) .. controls (-1.4,-2) and (1.4,-2) .. node [near start,below left=1mm] {$f^{-1}(b)$} (x0);
\draw[thick,dotted,->] (x1) -- (x1);
\draw[thick,dotted,->] (x1) .. controls (-1.3,0.3) and  (1,1.6) .. node [below] {$\ell_{-*}$} (x0);
\end{tikzpicture}}

\subsection{Morphic subshifts}\label{ss:morphic}
\newcommand\zero{{\mathbf0}}
\newcommand\one{{\mathbf1}}

This class of examples will be treated in much more detail in~\cite{bartholdi-m:actions2}; here is a brief summary.

Consider a substitution $\phi\colon \mathcal A\to \mathcal A^*$, for a finite alphabet $\mathcal A$, extended by concatenation to a map still written $\phi\colon \mathcal A^*\to \mathcal A^*$. We make the assumption that \emph{$\phi$ is length-increasing}, namely $|\phi^n(w)|\to\infty$ for all non-empty $w\in \mathcal A^*$. (There are fascinating questions about non-length-increasing substitutions, see~\cite{salo:quasiminimal}, to which I will return in another article\texttrademark.) The substitution $\phi$ is \emph{primitive} if there is some $n\in\N$ such that $\phi^n(a)$ contains every letter from $\mathcal A$, for all $a\in\mathcal A$. This is the most often considered case, and implies that $\phi$ is length-increasing, but we do \emph{not} make this assumption. Consider the space $\mathcal A^\Z$ of bi-infinite sequences over $\mathcal A$; it has the topology of a Cantor set. The substitution $\phi$ then also induces a map $\phi\colon \mathcal A^\Z\to \mathcal A^\Z$, such that $\phi(x)_0$ is the first letter of $\phi(x_0)$. There is a natural, continuous action of $\Z$ on $\mathcal A^\Z$, by shifting; we denote by $S\colon \mathcal A^\Z\to \mathcal A^\Z$ the generator of this action, defined by $S(x)_n=x_{n+1}$.

The \emph{substitutional subshift} associated with $\phi$ is defined as
\[\Omega_\phi=\bigcap_{n\ge0}\Z\cdot\phi^n(\mathcal A^\Z).\]
In other words, every element of $\Omega_\phi$ may be \emph{desubstituted} arbitrarily many times, after appropriate shifting. By the very definition, it is a closed, $\Z$-invariant subset of $\mathcal A^\Z$.

Consider now another finite alphabet $\mathcal B$, and a map $\psi\colon \mathcal A\to \mathcal B$. A \emph{morphic subshift} is a subshift $\Omega_{\phi,\psi}\subseteq \mathcal B^\Z$ that is the image of a substitutional subshift $\Omega_\phi\subseteq \mathcal A^\Z$ under $\psi$.

As a simple example (more appear below), consider the \emph{Golay-Rudin-Shapiro sequence}: $\mathcal A=\{a,b,c,d\}$ with
\[\phi(a)=a b,\quad\phi(b)=a c,\quad\phi(c)=d b,\quad\phi(d)=d c\]
and $\mathcal B=\{-,+\}$ with $\psi(a)=\psi(b)={+}$ and $\psi(c)=\psi(d)={-}$. The Golay-Rudin-Shapiro sequence~\cite{oeis}*{A020985} is the $\psi$-image of the fixed point $\phi^\infty(a)$,
\[+ + + - + + - + + + + - - - + - + + + - + + - + - - - + + + - + + + + - + + - +\dots\]
and $\Omega_{\phi,\psi}$ is the closure of its $\Z$-orbit. For more details see~\S\ref{ss:grs}

\begin{prop}\label{prop:morphic}
  Let $\Omega_{\phi,\psi}\subseteq \mathcal B^\Z$ be a morphic subshift. Then there is an $\omega$-regular language $L$ and a regular equivalence relation $\approx$ on $L$ such that $L/{\approx}$ is homeomorphic to $\Omega_{\phi,\psi}$, and such that the translation action of $\Z$ on $\Omega_{\phi,\psi}$ lifts to an automatic action on $L$.

  Furthermore, if $\phi$ is a primitive substitution and $\psi=1$, then the automatic presentation may be assumed to be injective (namely, the relation $\approx$ is the identity).
\end{prop}
This is very close to Vershik's theory of ``adic'' transformations~\cite{vershik-livshits:adic}; though the results are usually stated and proven in the measurable setting rather than the topological one. There is also a strong connection to Dumont-Thomas numeration systems and Frougny's automatic encoding of numbers~\cites{dumont-thomas:numeration,frougny:rnfa}. See also~\cite{durand-host-skau:substitutional}.

Here is a failed, but simpler, attempt at proving the proposition. Consider the case of a substitutional subshift $\Omega_\phi$, and define the following language $L$. Its alphabet is
\[\Sigma'=\{(a,i)\mid a\in \mathcal A,0\le i<|\phi(a)|\}.\]
$L$ is accepted by the automaton with stateset $\mathcal A$, all states being initial, and with an edge from $a\in \mathcal A$ to $b\in \mathcal A$ with label $(b,i)$ whenever $a$ is the $i$th letter (counted from $0$) in $\phi(b)$.

Consider $x=x_0x_1\dots\in L$, and the following sequence of pairs $(w,i)$ where $w$ is a word over $\mathcal A$ and $i\in\{1,\dots,|w|\}$ marks a letter of $w$. Write $x_k=(a_{k+1},i_k)$ and let $a_0$ be the $i_0$th letter of $\phi(a_1)$; then start with the word $w_0=a_0$, with the mark at position $0$. The next word is $w_1=\phi(a_1)$, which contains $a_0$ at position $i_0$. The next word is $w_2=\phi^2(a_2)$, and note that it contains $w_1$ as a factor (subword of consecutive letters) since the $i_1$th letter of $\phi(a_2)$ is $a_1$. Put the mark at position $i_0$ within that factor. At all steps, each $w_k$ is a factor of $w_{k+1}$, preserving the marking. Let $\overline x\in \mathcal A^\Z$ be the limit of these words $w_k$, viewed as a bi-infinite word centered at the mark. The map $x\mapsto\overline x$ is the desired bijection $L\to\Omega_\phi$.

There are a few problems with this construction, in that $\overline x$ could be only one-sided infinite; and the map $x\mapsto\overline x$ could fail to be injective (this is the problem of \emph{recognizability}, see~\cites{mosse:recognizability,bezuglyi+:aperiodic,berthe-s-t-y:recognizability}), or surjective. We proceed now with details.

\begin{proof}
  We begin with the case of a substitutional subshift $\Omega_\phi$, for a substitution $\phi$. Let $M$ be the maximal length of $\phi(a)$ over all $a\in \mathcal A$. Define the following language $L$, with alphabet
  \[\Sigma=\{(\ell\underline a r,i)\mid\ell,a,r\in \mathcal A,0\le i<M\}.\]
  $L$ is accepted by the automaton $\mathscr A$ with stateset $\Sigma$, all states being initial, and with an edge from $(\ell\underline a r,i)$ to $(m\underline b s,j)$ whenever $i<|\phi(b)|$ and $\ell a r$ is a factor of $\phi(m b s)$, with the `$a$' occurring at position $i$ (counted from $0$) relative to the beginning of $\phi(b)$ in $\phi(m b s)$. The label of the edge is the source vertex $(\ell\underline a r,i)$.

  Consider $x=x_0x_1\dots\in L$, and the following sequence of pairs $(w,i)$ where $w$ is a word over $\mathcal A$ and $i\in\{1,\dots,|w|\}$ marks a letter of $w$. Write $x_k=(\ell_k\underline{a_k}r_k,i_k)$ for all $k$, and start with the word $w_0=\ell_0\underline{a_0}r_0$, with the mark at position $1$, namely on the `$a_0$'. The next word is $w_1=\phi(\ell_1 a_1 r_1)$, which contains $w_0$ at position $i_0$. The next word is $w_2=\phi^2(\ell_2 a_2 r_2)$. At all steps, each $w_k$ is a factor of $w_{k+1}$, preserving the marking. Let $y$ be the limit of these words $w_k$, centered at the mark, and note that it is a bi-infinite word because $\phi$ is length-increasing. It is furthermore an infinitely iterated application of $\phi$, so belongs to $\Omega_\phi$. We have constructed a map $\pi\colon x\mapsto y$ from $L$ to $\Omega_\phi$.

  Conversely, given $y\in\Omega_\phi$, write it as $S^i(z)$ for some $i\in\Z$ and $z\in\Omega_\phi$. We may assume furthermore $0\le i<|\phi(z_0)|$. Construct a sequence $x\in\Sigma^\omega$ whose first letter $x_0$ is $(y_{-1}\underline{y_0}y_1,i)$, and whose remainder is constructed by applying the same procedure to $z$. We have constructed a preimage $x$ of $y$, so the map $\pi\colon L\to\Omega_\phi$ is surjective.

  We consider now a letter coding $\psi\colon \mathcal A\to \mathcal B$. In $L\times L$, consider the pairs of words $\dots((\ell_n\underline{a_n}r_n,i_n),(m_n\underline{b_n}s_n,j_n))\dots$ such that $\psi(a_n)=\psi(b_n)$ for all $n$. This defines an automatic relation $\approx$, and $x\approx x'$ if and only if $\psi(\pi(x))=\psi(\pi(y))$; so we have encoded $\Omega_{\phi,\psi}$ as $L/{\approx}$.

  If $\psi=1$ and $\phi$ is primitive, we quote~\cites{mosse:recognizability,bezuglyi+:aperiodic,berthe-s-t-y:recognizability} to assert that the substitution $\phi$ is \emph{recognizable}, so that the encoding $\pi$ is injective.

  We turn to the relation on $L\times L$ lifting the generator $S$ of $\Omega_{\phi,\psi}$. Its automaton $\mathscr S$ has alphabet $\Sigma^2$, stateset $\Sigma\sqcup \Sigma^2$; it contains a copy of $\mathscr A$ which accepts the diagonal of $L$ on the stateset $\Sigma$, and additional edges from $((\ell\underline a r,i),(a\underline r c,i+1))$ to $(m\underline b s,j)$ whenever $\ell a r c$ is a factor at position $i$ of $\phi(m b s)$, and from $((\ell\underline a r,|\phi(b)|-1),(a\underline r c,0))$ to $((m\underline b s,j),(b\underline{s}d,k))$ whenever $\ell a r c$ is a factor as position $|\phi(b)|-1$ of $\phi(m b s d)$; as before, the label of an edge is its source vertex. All states in $\Sigma^2$ are initial.

  The effect of $\mathscr S$ is to change the first symbol, if possible, by increasing the value of $i$. If its value increases beyond the length of the $\phi$-block that the mark belongs to, then the value of $i$ is reset to $0$ and the next symbol is incremented, with $\mathscr S$ continuing on the remainder of the sequence.
\end{proof}

Let us examine this construction in more detail for some examples.
\subsubsection{The Fibonacci subshift}\label{ss:fibo}
An instructive example of substitutional subshift is the following: consider the alphabet $\mathcal A=\{\zero,\one\}$ and the substitution $\phi\colon\zero\mapsto\zero\one,\one\mapsto\zero$.

The shift space $\Omega_\phi$ is the topological closure of the $\Z$-orbit of the fixed point $\phi^{2\infty}(\zero\underline\zero)=\dots\zero\one\zero\zero\one\zero\one\zero\underline\zero\one\zero\zero\one\zero\one\zero\dots$. It suffices, to encode sequences, to mark the previous letter, so we introduce the alphabet
\[a={}_\one\zero,\quad b={}_\zero\zero,\quad c={}_\zero\one,\]
on which the substition $\phi$ extends by $a\mapsto bc,b\mapsto ac,c\mapsto a$. The substitutional subshifts on $\{\zero,\one\}$ and $\{a,b,c\}$ are isomorphic, under the obvious map $a,b\mapsto\zero,c\mapsto\one$. We then select $\Sigma=\{1,2,3\}$ with
\[1=(a,0),\quad 2=(b,0),\quad 3=(c,1).\]
The language $L$ is accepted by the automaton
\[\begin{fsa}
    \node[state,initial above] (a0) at (0,0) {};
    \node[state,initial above] (b0) at (3,0) {};
    \node[state,initial above] (c1) at (-3,0) {};
    \draw (a0) edge[bend left=15] node {$1$} (b0) (b0) edge[bend left=10] node[above] {$2$} (a0)
    (a0) edge[bend left=10] node[above] {$1$} (c1)
    (c1) edge[bend right=25] node {$3$} (b0) (c1) edge[bend left=15] node {$3$} (a0);
  \end{fsa}
\]
as given by the proof of Proposition~\ref{prop:morphic}, after one removes the states that do not have infinite continuations. All states are residual. For example, $(12)^\omega$ encodes the fixed point $\xi=\phi^{2\infty}(\zero\underline\zero)$, while $(21)^\omega$ encodes the fixed point $\phi^{2\infty}(\one\underline\zero)$. The encoding of $\Omega_\phi$ by $L$ is bijective, so from now on we identify $\Omega_\phi$ and $L$. Note that the automaton for $L$ may be determinized as
\[\begin{fsa}
    \node[state,initial above] (s) at (0,1) {};
    \node[state] (a) at (0,0) {};
    \node[state] (b) at (3,0) {};
    \node[state] (c) at (-3,0) {};
    \draw (a) edge [bend left=15] node {$3$} (b);
    \draw (b) edge [bend left=10] node[above] {$1$} (a);
    \draw (b) edge [bend left=25] node[above] {$2$} (c);
    \draw (a) edge [bend left=10] node[above] {$2$} (c);
    \draw (c) edge [bend left=15] node[above] {$1$} (a);
    \draw (s) edge node[right] {$1$} (a);
    \draw (s) edge [bend right=10] node[above] {$2$} (c);
    \draw (s) edge [bend left=10] node[above] {$3$} (b);
  \end{fsa}\]
but we prefer to work with non-deterministic automata.

Projecting to the second coordinate $1,2\mapsto 0,3\mapsto 1$ in the alphabet yields a map $\Omega_\phi\to\{0,1\}^\omega$; and the point $S^n(\xi)$ projects, for $n\ge0$, to a sequence $e_0e_1\dots$ with finitely many $1$'s; and $n=\sum_{i\ge0} e_i F_{i+2}$ for $(F_i)$ the Fibonacci numbers $F_0=0,F_1=F_2=1,F_3=2,\dots$. Thus the orbit of $\xi$ closely corresponds to counting in base Fibonacci. Negative numbers $n$ correspond to sequences that are ultimately $(01)^\omega$, with $S^{-1}(\xi)$ projecting to $(01)^\omega$. The action of $S$ is automatic on $L$, given by the following automaton $\mathscr S$:
\[\mathscr S = \begin{fsa}[baseline=0]
    \node[state] (a0) at (0,1) {};
    \node[state] (b0) at (0,-1) {};
    \node[state] (c1) at (-3,0) {};
    \draw (a0) edge[bend left=15] node[right] {\small $\Delta1$} (b0) (b0) edge[bend left=10] node[left] {\small $\Delta2$} (a0)
    (a0) edge[bend left=15] node[above,sloped] {\small $\Delta1$} (c1)
    (c1) edge node[below,sloped] {\small $\Delta3$} (b0) (c1) edge[bend left=15] node[above,sloped] {\small $\Delta3$} (a0);
    \node[state,initial above] (x) at (3,0) {};
    \node[state,initial above] (y) at (6,0) {};
    \draw (x) edge [bend left=15] node[above] {\small $(1,2)$} (y);
    \draw (y) edge [bend left=15] node[below] {\small $(3,1)$} (x);
    \draw (x) edge node[below,sloped] {\small $(1,3)$} (b0) edge node[above,sloped] {\small $(2,3)$} (a0);
  \end{fsa}\]
in which all states are residual.

We could have started with a counting system, rather than a subshift! There is a counting system in which both positive and negative numbers are encoded by sequences over $\{0,1\}$ with finitely many $1$'s and no two consecutive $1$'s, the \emph{nega-Zeckendorf representation}: every $x\in\Z$ may uniquely be written, with these constraints, as $x=\sum_{n\ge0} x_n (-1)^n F_{n+1}$. Consider thus the language $L'\subset\{0,1\}^\omega$ consisting of sequences without two consecutive $1$'s:
\[\begin{fsa}
    \node[state,initial above] (a) at (0,0) {};
    \node[state,initial above] (b) at (3,0) {};
    \draw (a) edge [bend left=15] node {$0$} (b);
    \draw (b) edge [bend left=10] node[above] {$1$} (a);
    \draw (a) edge [loop left,looseness=25] node[left] {$0$} (a);
  \end{fsa}\]
and the automaton $\mathscr S'$ definining a bijection of $L'$, here deterministic:
\[\mathscr S' = \begin{fsa}[baseline=0]
    \node[state] (Tb) at (7,1) {};
    \node[state,initial right] (Ta) at (7,-1) {};
    \node[state] (Tc) at (5,-1) {};
    \node[state,initial above] (Td) at (5,1) {};
    \node[state] (1a) at (0,0) {};
    \node[state] (1b) at (3,0) {};
    \draw (Tb) edge [bend left=15] node {\small $(0,1)$} (Ta);
    \draw (Ta) edge [bend left=15] node {\small $(1,0)$} (Tb);
    \draw (Tb) edge node {$\Delta0$} (Td);
    \draw (Ta) edge node {$\Delta0$} (Tc);
    \draw (Tc) edge node[below,sloped] {\small $(1,0)$} (1b);
    \draw (Td) edge node[below,sloped] {\small $(0,1)$} (1b);
    \draw (1a) edge [bend left=15] node {\small $\Delta1$} (1b);
    \draw (1b) edge [bend left=15] node {\small $\Delta0$} (1a);
    \draw (1a) edge [loop below,looseness=25] node[below] {\small $\Delta0$} (1a);
  \end{fsa}
\]
in which all states are residual. It defines an automatic action of $\Z$, and implements the ``$+1$'' operation in nega-Zeckendorf representation. For example, $8=101010^\omega$ and $9=01010010^\omega$, and $(1,0)(0,1)(1,0)(0,1)(1,0)(0,0)(0,1)(0,0)^\omega$ is accepted. In fact, the two initial states define $+1$ and the two non-initial non-identity states define $-1$, and represent, after $n$ symbols have been read, a ``carry'' or ``borrow'' of $\pm F_{n+2}$ to be added.

For future reference, we include another automaton, defining the same transformation but in which every state knows the input and output symbol: the initial states are indicated as $t$ for the generator $S$, as $u$ for $S^{-1}$, and as $1$ for the identity:
\begin{equation}\label{eq:negazeckendorf}
  \begin{fsa}[baseline=0]
    \node[state] (t10) at (0,1) {$t$};
    \node[state] (u01) at (0,-1) {$u$};
    \node[state] (u00) at (2.5,1) {$u$};
    \node[state] (t00) at (2.5,-1) {$t$};
    \node[state] (t01) at (5,1) {$t$};
    \node[state] (u10) at (5,-1) {$u$};
    \node[state] (100) at (7,0) {$1$};
    \node[state] (111) at (9.5,0) {$1$};
    \draw (t10) edge [bend left=15] node {\small $(1,0)$} (u01)
          (u01) edge [bend left=15] node {\small $(0,1)$} (t10)
          (t10) edge node {\small $(1,0)$} (u00)
          (u01) edge node {\small $(1,0)$} (t00)
          (u00) edge node {\small $\Delta0$} (t01)
          (t00) edge node {\small $\Delta0$} (u10)
          (t01) edge node[above,sloped] {\small $(0,1)$} (100)
          (u10) edge node[above,sloped] {\small $(1,0)$} (100)
          (100) edge[loop above] node[above] {\small $\Delta0$} ()
          (100) edge [bend left=15] node {\small $\Delta0$} (111)
          (111) edge [bend left=15] node {\small $\Delta1$} (100);
  \end{fsa}
\end{equation}

In fact, these two dynamical systems are conjugate: consider the bijection between $L$ and $L'$ given by the following automaton $\mathscr T$:
\[\begin{fsa}
    \node[state] (a) at (0,4) {};
    \node[state,initial above] (b) at (3,4) {};
    \node[state] (c) at (6,4) {};
    \node[state] (d) at (9,4) {};
    \node[state] (e) at (3,2) {};
    \node[state] (f) at (6,2) {};
    \node[state] (g) at (9,2) {};
    \node[state,initial left] (h) at (3,0) {};
    \node[state] (i) at (6,0) {};
    \node[state] (j) at (9,0) {};
    \draw (a) edge[bend left=15] node {\small $(1,0)$} (b);
    \draw (b) edge[bend left=15] node {\small $(2,1)$} (a);
    \draw (b) edge node {\small $(3,0)$} (c);
    \draw (b) edge node {\small $(3,1)$} (e);
    \draw (c) edge[bend left=15] node {\small $\Delta1$} (d);
    \draw (d) edge[bend left=15] node {\small $(2,0)$} (c);
    \draw (d) edge[bend left=15] node {\small $(3,0)$} (g);
    \draw (e) edge node {\small $(2,0)$} (h);
    \draw (f) edge node {\small $(2,0)$} (c);
    \draw (g) edge[bend left=15] node {\small $(1,0)$} (d);
    \draw (g) edge node {\small $(2,0)$} (j);
    \draw (h) edge[bend left=15] node {\small $(1,0)$} (i);
    \draw (i) edge[bend left=15] node {\small $(2,0)$} (h);
    \draw (i) edge node {\small $(3,0)$} (f);
    \draw (i) edge[bend left=15] node {\small $(3,0)$} (j);
    \draw (j) edge[bend left=15] node {\small $\Delta1$} (i);
  \end{fsa}\]
It is easy to check that $\mathscr T$ defines indeed a bijection $L\to L'$, and $\mathscr S'\circ\mathscr T=\mathscr T\circ\mathscr S$.

\subsubsection{The Golay-Rudin-Shapiro subshift}\label{ss:grs}
The \emph{Golay-Rudin-Shapiro sequence} is the right-infinite word $w=w_0w_1\dots$ with all $w_n\in\{\pm1\}$ defined as follows: write $n=\sum e_i2^i$ in base $2$, and set $w_n=(-1)^{\sum e_i e_{i+1}}$; it computes thus the parity of the number of consecutive $1$'s in the base-$2$ expansion of $n$. Set $\mathcal A=\{a,b,c,d\}$ and define $\phi\colon\mathcal A\to\mathcal A^2$ by
\[\phi(a)=a b,\quad\phi(b)=a c,\quad\phi(c)=d b,\quad\phi(d)=d c;\]
set then $\mathcal B=\{-1,+1\}$ and define $\psi\colon\mathcal A\to\mathcal B$ by
\[\psi(a)=\psi(b)=+1,\quad\psi(c)=\psi(d)=-1.\]
It is easy to check $w=\psi(\phi^{2\infty}(a))$. We denote by $\Omega_{\phi,\psi}$ the corresponding morphic subshift; it is the closure of the $\Z$-orbit of $\psi(\phi^{2\infty}(c\underline a))$, or of any other starting seed $c\underline d$, $b\underline a$, $b\underline d$.

Consider the alphabet $\Sigma=\{1=(a,0),2=(b,1),3=(c,1),4=(d,0)\}$ and the language $L\subset\Sigma^\omega$ accepted by the automaton
\[\begin{fsa}
    \node[state,initial above] (a0) at (0,0) {};
    \node[state,initial above] (b1) at (3,0) {};
    \node[state,initial above] (c1) at (6,0) {};
    \node[state,initial above] (d0) at (9,0) {};
    \draw (a0) edge[loop left] node[left] {$1$} () edge[bend left=15] node {$1$} (b1);
    \draw (b1) edge[bend left=15] node {$2$} (a0) edge[bend left=15] node {$2$} (c1);
    \draw (c1) edge[bend left=15] node {$3$} (b1) edge[bend left=15] node {$3$} (d0);
    \draw (d0) edge[loop right] node[right] {$4$.} () edge[bend left=15] node {$4$} (c1);
  \end{fsa}
\]
The language $L$ is in bijection with $\Omega_\phi$, so there is a surjective map $L\twoheadrightarrow\Omega_{\phi,\psi}$ with kernel
\[\begin{fsa}
    \node[state,initial above] (e) at (0,0) {};
    \draw (e) edge[loop right,looseness=25] node[right] {$(\{1,2\}\times\{1,2\})\cup(\{3,4\}\times\{3,4\})$.} (e);
  \end{fsa}
\]
Note that no collaring is necessary in this case. The shift $S$ on $\Omega_\phi$, and therefore on $\Omega_{\phi,\psi}$, is represented by the following automaton $\mathscr S$:
\[\begin{fsa}
    \node[state] (a0) at (0,0) {};
    \node[state] (b1) at (3,0) {};
    \node[state] (c1) at (6,0) {};
    \node[state] (d0) at (9,0) {};
    \draw (a0) edge[loop left] node[left] {\small $\Delta1$} (a0) edge[bend left=15] node {\small $\Delta1$} (b1);
    \draw (b1) edge[bend left=15] node {\small $\Delta2$} (a0) edge[bend left=15] node {\small $\Delta2$} (c1);
    \draw (c1) edge[bend left=15] node {\small $\Delta3$} (b1) edge[bend left=15] node {\small $\Delta3$} (d0);
    \draw (d0) edge[loop right] node[right] {\small $\Delta4$} (d0) edge[bend left=15] node {\small $\Delta4$} (c1);
    \node[state,initial above] (a0s) at (0,2) {};
    \node[state,initial above] (b1s) at (3,2) {};
    \node[state,initial above] (c1s) at (6,2) {};
    \node[state,initial above] (d0s) at (9,2) {};
    \draw (a0s) edge[loop right] node {$(2,1)$} () edge node[right] {$(1,2)$} (a0);
    \draw (b1s) edge[loop right] node {$(2,4)$} () edge node[right] {$(1,3)$} (b1);
    \draw (c1s) edge[loop right] node {$(3,1)$} () edge node[right] {$(4,2)$} (c1);
    \draw (d0s) edge[loop right] node {$(3,4)$} () edge node[right] {$(4,3)$} (d0);
  \end{fsa}
\]

\subsubsection{A non-minimal example}\label{ss:nonminimal}
One of my motivations for the automaton formalism, compared to existing tools, is the possibility to encode non-minimal subshifts, and non-primitive substitutions. Consider for example $\mathcal A=\{\zero,\one\}$ and $\phi\colon\zero\mapsto\zero\zero,\one\mapsto\one\zero\one$. It appears for example in~\cite{salo:quasiminimal}*{Example~4}. I am grateful to Ruiwen Dong for help in understanding this example.

One has $\phi^\infty(\zero\underline\zero)=\zero^\Z$ and $\phi^\infty(\zero\underline\one)=\dots\zero\zero\underline\one\zero\one\zero^2\one\zero\one\zero^4\one\zero\one\dots$. Note that $\Omega_\phi$ is \emph{not} the set of words in $\mathcal A^\Z$ all of whose factors are factors of $\phi^n(a)$ for some $a\in A,n\in\N$, as is the case for primitive substitutions. Indeed $\phi^\infty(\one\underline\one)$ belongs to $\Omega_\phi$, but it has $\one\one$ as a factor, while $\phi^\infty(\one)$ does not.

We first enrich the alphabet $\mathcal A$ by ``collaring'': decorating each $\zero$ symbol with the previous and next symbol. We then set $\Sigma=\{1,2,3,4,5,6,7,8,9\}$ with
\begin{gather*}
  1=({}_\zero\zero_\zero,0),\quad 2=({}_\one\zero_\zero,0),\quad 3=({}_\zero\zero_\zero,1),\quad 4=({}_\zero\zero_\one,1),\\
  5=({}_\zero\one_\zero,0),\quad 6=({}_\one\one_\zero,0),\quad 7=({}_\one\zero_\one,1),\quad 8=({}_\zero\one_\zero,2),\quad 9=({}_\zero\one_\one,2).
\end{gather*}
Consider the language $L\subset\Sigma^\omega$ accepted by the automaton
\[\begin{fsa}
    \node[state,initial above] (zzz) at (0,0) {};
    \node[state,initial below] (ozz) at (2.5,-1) {};
    \node[state,initial above] (zzo) at (2.5,1) {};
    \node[state,initial above] (ozo) at (5,0) {};
    \node[state,initial above] (zoz) at (7.5,0) {};
    \node[state,initial below] (ooz) at (10,-1) {};
    \node[state,initial above] (zoo) at (10,1) {};
    \draw (zzz) edge[loop left] node[left] {$1,3$} () edge node {$1$} (zzo) edge node {$3$} (ozz);
    \draw (ozz) edge[loop above] node[above] {$2$} () edge node {$2$} (ozo);
    \draw (zzo) edge[loop below] node[below] {$4$} () edge node {$4$} (ozo);
    \draw (ozo) edge node {$7$} (zoz) edge[bend right=20] node[below,near start] {$7$} (ooz) edge[bend left=20] node[above,near start] {$7$} (zoo);
    \draw (zoz) edge[loop below] node {$5,8$} () edge node[below right] {$5$} (zoo) edge node[above right] {$8$} (ooz);
    \draw (ooz) edge[loop right] node[right] {$6$} ();
    \draw (zoo) edge[loop right] node[right] {$9$} ();
  \end{fsa}
\]
It contains multiple sublanguages, for example $\{5,8\}^*(59^\omega+86^\omega)$ in bijection with the $\Z$-orbit of $\phi^\infty(\one\underline\one)$. The fixed point $\zero^\Z$ is encoded by $\{1,3\}^\omega\cong\Z_2$; this is the only point of $\Omega_\phi$ that admits more than one encoding; so the equivalence relation
\[\begin{fsa}[every edge/.style={sloped,draw}]
    \node[state,initial above] (zzz) at (0,0) {};
    \node[state,initial below] (ozz) at (2.5,-1) {};
    \node[state,initial above] (zzo) at (2.5,1) {};
    \node[state,initial above] (ozo) at (5,0) {};
    \node[state,initial above] (zoz) at (7.5,0) {};
    \node[state,initial below] (ooz) at (10,-1) {};
    \node[state,initial above] (zoo) at (10,1) {};
    \draw[every edge/.style={draw}] (zzz) edge[loop left] node[left] {\small $\{1,3\}\times\{1,3\}$} ();
    \draw (zzz) edge node {\small $\Delta1$} (zzo) edge node[below left] {\small $\Delta3$} (ozz);
    \draw (ozz) edge[loop above] node[above] {\small $\Delta2$} () edge node[below right] {\small $\Delta2$} (ozo);
    \draw (zzo) edge[loop below] node[below left] {\small $\Delta4$} () edge node {\small $\Delta4$} (ozo);
    \draw (ozo) edge node {\small $\Delta7$} (zoz) edge[bend right=25] node[below,near start] {\small $\Delta7$} (ooz) edge[bend left=25] node[above,near start] {\small $\Delta7$} (zoo);
    \draw (zoz) edge[loop below] node[below=-1mm] {$\Delta5,\Delta8$} () edge node[below right] {\small $\Delta5$} (zoo) edge node[above right] {\small $\Delta8$} (ooz);
    \draw[every edge/.style={draw}] (ooz) edge[loop right] node[right] {\small $\Delta6$} ();
    \draw[every edge/.style={draw}] (zoo) edge[loop right] node[right] {\small $\Delta9$} ();
  \end{fsa}
\]
is the kernel of the map $L\twoheadrightarrow\Omega_\phi$.

The system $\Omega_\phi$ is not irreducible; besides the fixed set $\{\zero^\Z\}$ it also contains for example the closure of the $\Z$-orbit of $\phi^\infty(\zero\underline\one)$, which is coded by $L\cap\{1,\dots,7\}^\omega$.

The generator $S$ of the shift is coded by the following automaton:
\[\begin{fsa}[every edge/.style={sloped,draw}]
    \node[state] (zzz) at (0,0) {};
    \node[state] (ozz) at (2,-1) {};
    \node[state] (zzo) at (2,1) {};
    \node[state] (ozo) at (4,0) {};
    \node[state] (zoz) at (6,0) {};
    \node[state] (ooz) at (5,-1.732) {};
    \node[state] (zoo) at (5,1.732) {};
    \draw (zzz) edge[loop below] node[below] {\small $\Delta1,\Delta3$} () edge node {\small $\Delta1$} (zzo) edge node[above] {\small $\Delta3$} (ozz);
    \draw (ozz) edge[loop above] node[above] {\small $\Delta2$} () edge node[above] {\small $\Delta2$} (ozo);
    \draw (zzo) edge[loop above] node[right] {\small $\Delta4$} () edge node {\small $\Delta4$} (ozo);
    \draw (ozo) edge node {\small $\Delta7$} (zoz) edge node[below] {\small $\Delta7$} (ooz) edge node[above] {\small $\Delta7$} (zoo);
    \draw[every edge/.style={draw}] (zoz) edge[loop right] node {$\Delta5,\Delta8$} ();
    \draw (zoz) edge node[above] {\small $\Delta5$} (zoo) edge node[above] {\small $\Delta8$} (ooz);
    \draw (ooz) edge[loop below] node[right] {\small $\Delta6$} ();
    \draw (zoo) edge[loop above] node[left] {\small $\Delta9$} ();
    \node[state,initial above] (s0) at (-2,0) {};
    \draw[every edge/.style={draw}] (s0) edge [loop left] node[left] {\small $(3,1)$} ();
    \draw (s0) edge node[above] {\small $(1,3)$} (zzz)
      edge[bend right=20] node[below] {\small $(1,2)$} (ozz) edge[bend left=20] node[above] {\small $(1,4)$} (zzo)
      edge[bend right=60] node[below] {\small $(2,4)$} (ozo);
    \node[state,initial right] (s1) at (7,-1.732) {};
    \node[state,initial right] (s2) at (7,1.732) {};
    \draw (s1) edge[loop below] node[right] {$(8,2)$} () edge node {\small $(5,7)$} (zoz);
    \draw (s2) edge[loop above] node[right] {$(4,5)$} () edge node {\small $(7,8)$} (zoz) edge node {\small $(7,9)$} (zoo);
    \node[state,initial below] (s3) at (3,-2) {};
    \draw[every edge/.style={draw}](s3) edge[loop left] node[left] {$(9,6)$} ();
  \end{fsa}
\]
Note that the two rightmost states define, when restricted to $\{1,3\}^\omega$, a copy of the binary odometer.

\subsection{A circle rotation}
We give one example of an action involving an equivalence relation: an irrational rotation on the circle. Consider $\phi=(\sqrt5-1)/2\approx0.618$, and the action of $\Z$ on $X=\R/\Z$ by $n\cdot x=x+n\phi$. This action is automatic, as follows.

The language $L$ is the same language $L'$ from~\S\ref{ss:fibo}, namely all words over $\{0,1\}$ without two consecutive $1$'s. The coding map $\pi\colon L'\twoheadrightarrow X$ is the base-$\phi$ expansion:
\[\pi(x)=\sum_{n\ge0}x_n\phi^{n+1}\mod\Z.\]

An almost-inverse of $\pi$ is given as follows. Consider two maps $t_0,t_1$, whose union is a map $[0,1)\to[0,1)$, given by
\[t_0\colon
  \begin{cases}[0,\phi)&\mapsto[0,1)\\
    x&\mapsto x/\phi
  \end{cases},\qquad
  t_1\colon
  \begin{cases}
    [\phi,1)&\mapsto[0,1)\\
    x&\mapsto (x-\phi)/\phi
  \end{cases},
\]
and assign to every $x\in[0,1)$ its unique itinerary under $\{t_0,t_1\}$. This recovers one element of $\pi^{-1}(x)$; the possibly other element of $\pi^{-1}(x)$ arises from the same formulas for $t_0,t_1$, but with the other choice of endpoints $(0,\phi],(\phi,1]$. We have $1=\phi+\phi^3+\dots$, so the kernel of $\pi$ is the relation
\[\begin{fsa}
[every edge/.style={draw},every state/.style={draw,circle,minimum size=6mm}]
    \node[state,initial above] (zero) at (0,0) {$0$};
    \node[state,initial left] (zero') at (0,-2) {$0$};
    \node[state,initial above] (mone) at (2,0) {\clap{$-1$}};
    \node[state] (mphi) at (4,0) {\clap{$-\phi$}};
    \node[state,initial above] (one) at (-2,0) {$1$};
    \node[state] (phi) at (-4,0) {\clap{$\phi$}};
    \draw (zero) edge[bend left=15] node[right] {\small $\Delta1$} (zero')
    (zero') edge[bend left=15] node[left] {\small $\Delta0$} (zero)
    (zero') edge[loop right] node[right] {\small $\Delta0$} ()
    (zero) edge node[above] {\small $(0,1)$} (mone)
    (zero) edge node[above] {\small $(1,0)$} (one)
    (mone) edge[bend left=15] node[above] {\small $(1,0)$} (mphi)
    (mphi) edge[bend left=15] node[below] {\small $(0,0)$} (mone)
    (one) edge[bend right=15] node[above] {\small $(0,1)$} (phi)
    (phi) edge[bend right=15] node[below] {\small $(0,0)$} (one);
  \end{fsa}
\]
For example, we have $\pi(0^\omega)=0=1=\pi((10)^\omega)$, and $\pi((01)^\omega)=\phi=10^\omega$. The values indicated inside a state represents the value $\sum_{i\ge0}(a_i-b_i)\phi^{i+1}$ for all sequences $(a_0,b_0)(a_1,b_1)\dots$ accepted from that state.

The action of $\Z$ on $X$ lifts to the action of $\Z$ on $L'$ given in~\S\ref{ss:fibo}, namely adding integers in their nega-Zeckendorff representation, or acting on the Fibonacci subshift. Indeed the language $L'$ could have been obtained as follows: in $X$, consider the orbit $\phi\Z\bmod\Z$ of $0$, which is dense because $\phi$ is irrational. Replace each point $n\phi$ by two points $(n\phi)^-$ and $(n\phi)^+$, with the order topology (so we think of a big gap between $(0\phi)^-$ and $(0\phi)^+$, and smaller and smaller gaps as $n$ increases). The resulting topological space $\widetilde X$ is precisely $L'$, and the action of $\Z$ on $X$ obviously lifts to an action on $\widetilde X$, by $n\cdot(m\phi)^{\pm}=((n+m)\phi)^\pm$.

I am not aware of any ``interesting'' example of an automatic group action on a quotient $L/{\cong}$ for which the relation $\cong$ is not closed; see~\S\ref{ss:bounded} where this question appears.

\section{Bounded automata}\label{ss:bounded}
Consider a regular language $L\subseteq \Sigma^\omega$, recognized by a chosen automaton $\mathscr L$. In this section, we assume that $L$ is closed, and that $\mathscr L$ is a safety automaton (namely, all its states are accepting). Let $\mathscr A$ be a safety automaton on alphabet $\Sigma\times\Sigma$, and assume without loss of generality that $\mathscr A$ contains as a subautomaton $\Delta(\mathscr L)$, the automaton $\mathscr L$ in which every label $\sigma\in\Sigma$ is replaced by $(\sigma,\sigma)\in\Sigma\times\Sigma$.
\begin{defn}
  A right-infinite path in $\mathscr A$ is called \emph{singular} if it does not reach $\Delta(\mathscr L)$.
  
  A state $s$ of $\mathscr A$ is called \emph{bounded} if the number of singular paths that start at $s$ is finite. It is called \emph{finitary} if there are no such right-infinite paths.

  The automaton $\mathscr A$ is called \emph{bounded} if all its initial states are bounded, and \emph{finitary} if all its initial states are finitary.
\end{defn}

We shall be particularly interested in partial bijections $g\colon L\dashrightarrow L$ represented by bounded automata, which we call \emph{bounded partial bijections}.

\begin{prop}
  The product and inverse of bounded partial bijections are bounded; therefore an inverse semigroup consists of bounded elements as soon as its generators are bounded.
\end{prop}  
\begin{proof}
  If $\{(s_i,t_i)\}$ is the set of singular paths of $\mathscr A$ and $\{(u_j,v_j)\}$ is the set of singular paths of $\mathscr B$, then $\{(t_i,s_i)\}$ is the set of singular paths of $\mathscr A^{-1}$. The composition of the safety automata $\mathscr A\mathscr B$ is the automaton with stateset $V(\mathscr A)\times V(\mathscr B)$ and transitions $u\to w$ with label $(x,z)$ for all transitions $u\to v$ with label $(x,y)$ in $\mathscr A$ and transitions $v\to w$ with label $(y,z)$ in $\mathscr B$. Therefore, the set of singular paths of $\mathscr A\mathscr B$ is contained in $\{(s_i,t_i)\}\cup\{(u_j,v_j)\}\cup\{(s_i,v_j)\mid t_i=u_j\}$.
\end{proof}

It is easy to verify, given an automaton, whether it defines a bounded transformation:
\begin{lem}
  Let $\mathscr A$ be an automaton defining a partial bijection $g\colon L\dashrightarrow L$, and assume that $\mathscr A$ is minimized. The following are equivalent:
  \begin{enumerate}
  \item There is a bounded number of paths of length $n$ in $\mathscr A$ and not reaching $\Delta(\mathscr L)$;
  \item The transformation $g$ is bounded;
  \item There are finitely many left-infinite paths in $\mathscr A$ not reaching $\Delta(\mathscr L)$;
  \item There are finitely many bi-infinite paths in $\mathscr A$ not reaching $\Delta(\mathscr L)$;
  \item All cycles in $\mathscr A\setminus\Delta(\mathscr L)$ are disjoint.
  \end{enumerate}
\end{lem}
\begin{proof}
  This is really a statement about the directed graph $\mathscr A\setminus\Delta(L)$. We show that all conditions are equivalent to the last one.

  If all cycles are disjoint, then every bi-infinite path has to follow on one of these cycles; right-infinite paths may exit it at some point, but then have finitely many possibilities before ending; left-infinite paths have finitely many possibilities before entering such a cycle; and finite paths have a bounded number of possibilities before and after cycling around a cycle.

  Conversely, if there are two cycles that are connected to each other, then there is a path $P^\omega Q R^\omega$, leading to infinitely many bi-infinite paths by choosing their origin; they have infinitely many left-infinite and right-infinite truncations, and an unbounded number of length-$n$ factors.
\end{proof}

Examples of bounded automata abound, in particular coming as iterated monodromy groups (see~\S\ref{ss:img}). The Grigorchuk group is also bounded, as are most examples from the literature.
\begin{prop}
  If $f$ is a post-critically finite complex polynomial (namely, every critical point of $f\in\C[z]$ has a preperiodic or periodic orbit), then the iterated monodromy group of $f$ is bounded.\qed
\end{prop}

The following result extends the idea that ``carries only occur at $\dots 999$'' in base $10$, namely that, when an element of a subshift is decomposed into superblocks for the substitution, the action of the generator $+1$ of $\Z$ is non-finitary only when at the rightmost boundary of infinitely many superblocks:
\begin{prop}
  If $\Omega_{\phi,\psi}$ is a morphic subshift, then the translation action of $\Z$ on $\Omega_{\phi,\psi}$ can be assumed to be bounded.
\end{prop}
\begin{proof}
  This follows directly from the construction: the only cycle, in the automaton we gave for the generator $S$ of the shift, is the one along transitions $(\ell\underline a r,i)\to(m\underline b s,j)$ for which $i=|\phi(b)|-1$.
\end{proof}
(Note that we are under the assumption that $\Omega_{\phi,\psi}$ admits an injective encoding by $L$. However, the transitions above, and given in the proof of Proposition~\ref{prop:morphic}, always define a bounded action of $\Z$ on $L$. See Example~\ref{ss:nonminimal} for a morphic subshift with a fixed point, which ``explodes'' into a copy of $\Z_2$ in $L$.)

\subsection{Self-similar families of inverse semigroups}
Let $G$ be an inverse semigroup of partial bijections of $L\subseteq\Sigma^\omega$, acting automatically. For $g\in G$ and $u,v\in\Sigma^*$ of the same length, the \emph{state} $g@(u,v)$ of $g$ at $(u,v)$ is the relation on $\Sigma^\omega$ defined by
\[(x,y)\in g@(u,v)\Longleftrightarrow (u x)\cdot g=v y.\]
We note that, for the state $g@u$ defined in~\S\ref{ss:autgroup} for actions on regular trees, $g@u=g@(u,u\cdot g)$ is again a tree isometry; while in general $g@(u,v)$ could be a partial isometry even if $g$ is total. We have the following useful formula for composition:
\[(gh)@(u,v)=\bigcup_{|w|=|u|=|v|}g@(u,w)h@(w,v).\]

\begin{defn}
  A partial bijction $g\colon L\dashrightarrow L$ is \emph{finite state} if the set of states $\{g@(u,v):u,v\in\Sigma^*,|u|=|v|\}$ is finite.
\end{defn}
If $g$ is an automatic transformation, then it is finite state; the states $g@(u,v)$ are also automatic, with same underlying automaton as $g$ except that the initial states are now those states reachable from an original initial state after reading $(u,v)$.

Conversely, if $g$ is finite state, then it is automatic, and may be represented by the following automaton $\mathscr A_g$: its stateset is $\{g@(u,v):u,v\in\Sigma^*,|u|=|v|\}$, its initial stateset is $\{g@(\varepsilon,\varepsilon)\}$, and there is a transition from $g@(u,v)$ to $g@(u x,v y)$ labeled $(x,y)$ for all $u,v\in\Sigma^*,x,y\in\Sigma$ whenever $g@(u x,v y)$ is non-empty. In fact, $\mathscr A_g$ is the minimal automaton representing $g$.

For $u\in\Sigma^*$, let $u^{-1}L$ denote the set $\{x\in\Sigma^\omega\mid u x\in L\}$, and note that $g@(u,v)$ is a partial bijection $u^{-1}L\dashrightarrow v^{-1}L$. Indeed it is the composition of maps
\[u^{-1}L\overset{u}\longrightarrow L\overset{g}\dashrightarrow L\overset{v^{-1}}\dashrightarrow v^{-1}L,\]
where the first map is $x\mapsto u x$ and the last map is $v x\mapsto x\colon L\cap v\Sigma^\omega\to v^{-1}L$.

Note that if $L$ is an $\omega$-regular language then there are finitely many languages of the form $u^{-1}L$; indeed, these may be taken as the states of an automaton defining $L$. Let us enumerate them as $L_0=L,L_1,\dots,L_n$.
\begin{defn}
  A family of sets of partial bijections $\{G_{i,j}\mid 0\le i,j\le n\}$, with $G_{i,j}\colon L_i\to L_j$, is called \emph{self-similar} if for all $g\in G_{i,j}$ and all $u,v\in\Sigma^*$ of same length one has $g@(u,v)\in G_{i',j'}$ where $i',j'$ are determined by $u^{-1}L_i=L_{i'}$ and $v^{-1}L_j=L_{j'}$. They form a \emph{self-similar inverse semigroup family} if furthermore $G_{i,j}^{-1}=G_{j,i}$ and $G_{i,j}G_{j,k}\subseteq G_{i,k}$ for all $i,j,k$.
\end{defn}

Starting from an inverse semigroup $G$ acting on $L$, one defines
\begin{equation}\label{eq:Gij}
  G_{i,j}=\{g@(u,v)\mid L_i=u^{-1}L,L_j=v^{-1}L\}
\end{equation}
to obtain a self-similar inverse semigroup family; and conversely from a self-similar inverse semigroup family $(G_{i,j})$ we obtain a semigroup $G_{0,0}$ acting on $L$. Performing these two steps starting with an inverse semigroup $G$ acting on $L$, one obtains a possibly larger semigroup acting on $L$, the \emph{state closure} of $G$.

\begin{defn}
  A self-similar inverse semigroup family $(G_{i,j})_{0\le i,j\le n}$ is \emph{contracting} if there is a family of finite subsets $\mathscr N_{i,j}\subseteq G_{i,j}$ such that the following holds: for every $g\in G_{i,j}$ there is $n\in\N$ such that $g@(u,v)\in\bigcup_{0\le i',j'\le n}\mathscr N_{i',j'}$ whenever $|u|=|v|\ge n$.

  A self-similar inverse semigroup family $(G_{i,j})$ is \emph{finitely generated} if there is a family of finite subsets $S_{i,j}\subseteq G_{i,j}$ such that every self-similar inverse semigroup family that termwise contains $(S_{i,j})$, termwise contains $(G_{i,j})$. It is \emph{finite} if all $G_{i,j}$ are finite.

  A self-similar inverse semigroup family $(G_{i,j})$ is \emph{bounded} if every $g\in \bigcup_{i,j}G_{i,j}$ is a bounded partial bijection.
\end{defn}
\begin{lem}
  If $G$ is a finitely generated inverse semigroup acting automatically, then its associated self-similar inverse semigroup family $(G_{i,j})$ as in~\eqref{eq:Gij} is finitely generated.
\end{lem}
\begin{proof}
  Let $S$ be a generating set of $G$, and represent each element of $S$ by an automaton. Let $L$ be the language on which $G$ acts, and write $\{L_0,\dots,L_n\}$ for the states of $L$. Define $S_{i,j}$ to be the set of partial bijections $s@(u,v)$ with $s\in S$ and $u,v\in\Sigma^*$ two words of same length with $L_i=u^{-1}L,L_j=v^{-1}L$. Each such $s@(u,v)$ may be represented by an automaton, namely the automaton defining $s$ but for which the initial states are now all those states reachable from an original initial state after reading $(u,v)$. There are finitely many possibilities for such $s@(u,v)$.
\end{proof}

\begin{prop}[see~\cite{bondarenko-n:pcf}*{Theorem~5.3}]\label{prop:boundednuk}
  If $G=(G_{i,j})$ is a finitely generated bounded self-similar inverse semigroup family, then it is contracting.

  More precisely, let $\ell$ be the least common multiple of cycle lengths in automata among a finite generating set of $G$, and let $H=(H_{i,j})_{0\le i,j\le n}$ be the inverse semigroup family generated by all finitary states of these automata. Then $H$ is finite, and the nucleus of $G$ consists of elements of $H$ and of automata that contain a single cycle of length dividing $\ell$, with elements of $H$ attached to it.
\end{prop}
\begin{proof}
  Let $(S_{i,j})$ be a finite generating set for $G$, and represent each $s\in S_{i,j}$ by an automaton, for example its minimal one.

  Recall that these automata have the following structure: all states are either \emph{finitary} (namely leading to $\Delta(\mathscr L)$), \emph{periodic} (namely lying on one of the disjoint cycles of the automaton), or \emph{preperiodic} (the remaining case). There exists a constant $m$ such that, for every preperiodic state $s$, all states $s@(u,v)$ are periodic or finitary if $|u|=|v|\ge m$. Define therefore
  \[S'_{i',j'}=\{s@(u,v):s\in S_{i,j},|u|=|v|=m,u^{-1}L_i=L_{i'},v^{-1}L_j=L_{j'}\},\]
  and let $(G'_{i,j})$ be the self-similar inverse semigroup family that it generates. The nucleus of $G'$ and of $G$ coincide, so it suffices to prove that $G'$ is contracting. We may therefore assume that all generators of $G$ are periodic or finitary.

  Letting as in the statement $\ell$ denote the least common multiple of the lengths of the cycles of periodic elements among the $S_{i,j}$, we obtain that, for every periodic state $s$ and every $u,v\in\Sigma^\ell$, either $s@(u,v)=s$ or $s@(u,v)$ is finitary; and there exists a single pair $(u,v)$ such that $s@(u,v)=s$.

  For all $i,j$, let $H_{i,j}$ be the self-similar inverse semigroup family generated by $(S_{i,j}\cap\text{finitary})$, and note that all its elements are finitary partial bijections. Let $h\in\N$ be such that $s@(u,v)\subset\Delta(\Sigma^\omega)$ for all $s\in H_{i,j}$ and all $u,v\in\Sigma^{h\ell}$.
  
  For all $i,j$, let $\mathscr N'_{i,j}$ denote the elements $g\in G_{i,j}$ such that there exists unique $u,v\in\Sigma^\ell$ with $g@(u,v)=g$, and $g@(u',v')\in\bigcup_{i',j'}H_{i',j'}$ for all other $(u',v')$. Since $H$ is finite, there are finitely many choices for $u,v$ and the $g@(u',v')$, so $\mathscr N'_{i,j}$ is finite.

  Consider now $g\in G_{i,j}$, and write it as a product $g=g_1\cdots g_{t'}$ with each $g_k\in S_{i_{k-1},i_k}$ and $i_0=i,i_{t'}=j$, and with a minimal number $t$ of periodic $g_k$'s. We will prove by induction on $t$ that there exists $n\in\N$ with $g@(u,v)\in\bigcup_{i,j}\mathscr N'_{i,j}\cup S_{i,j}$ whenever $|u|=|v|=n\ell$.

  If $t=1$, that is all $g_k$ are finitary except one, then $g=q s r$ with $q,r\in\bigcup_{i,j} H_{i,j}$ and $s$ periodic. There exists therefore $n\in\N$ such that $q@(u,v),r@(u,v)\subset\Delta(\Sigma^\omega)$ whenever $|u|=|v|=n\ell$. In particular, $q$ and $r$ act as partial bijections of $\Sigma^{n\ell}$, still written $q,r$ respectively; and $g@(u,v)=s@(u\cdot q,v\cdot r^{-1})$, so $g@(u,v)\in\bigsqcup_{i,j} S_{i,j}$.

  Suppose that the claim is proven for all elements that may be written with $<t$ periodic generators, and consider an element $g=g_1\cdots g_{t'}$ which requires $t$ periodic generators. For every $k$ and every $u,v\in\Sigma^{h\ell}$, either $g_k@(u,v)=g_k$ or $g_k@(u,v)$ is finitary; and furthermore if $g_k$ was finitary then $g_k@(u,v)\subset\Delta(\Sigma^\omega)$. Therefore, either the periodic generators appear in more than one $g@(u,v)$ and therefore all $g@(u,v)$ can be written with $<t$ periodic generators, in which case induction applies; or all periodic generators of $g$ appear in some given $g@(u,v)$. In that case, the non-periodic generators also vanished, so $g@(u,v)$ is a product of $t$ periodic generators $g'_1\cdots g'_t$.

  Applying again the argument to $g@(u,v)$, we obtain that for every $k$ and every $u',v'\in\Sigma^\ell$, either $g'_k@(u',v')=g'_k$ or $g'_k@(u',v')$ is finitary. If the periodic generators appear in more than one $g@(u u',v v')$, then all $g@(u u',v v')$ can be written with $<t$ periodic generators, in which case induction applies; or all periodic generators of $g@(u,v)$ appear in some given $g@(u u',v v')$, in which case $g@(u u',v v')=g@(u,v)$ so $g@(u,v)\in\bigcup_{i,j}\mathscr N'_{i,j}$.

  Now set $\mathscr N_{i,j}=\bigcup_{|u|=|v|<\ell}\mathscr N'_{i,j}@(u,v)\cup S_{i,j}$. It follows that the nucleus of $G$ is contained in $\bigcup_{i,j}\mathscr N_{i,j}$.
\end{proof}

Though we shall not need this here, Nekrashevych's theory of ``limit spaces'' extends quite directly to the inverse semigroup setting. Recall that for a group $G$ acting by tree isometries on $\Sigma^*$ in a self-similar and contracting manner, its associated \emph{limit space} is the quotient of $\Sigma^{-\N}$ by ``asymptotic equivalence'': this is the equivalence relation on $\Sigma^{-\N}$ defined by
\[(x_n)_{n<0}\sim(y_n)_{n<0}\text{ if and only if }\exists(g_n)_{n<0}:\begin{cases}\forall n: g_n(x_n\dots x_{-1})=y_n\dots y_{-1},\\\#\{g_n:n\in-\N\}<\infty.\end{cases}\]
In fact, it is enough to consider elements $g_n$ in the nucleus.

In our setting, there is no action on $\Sigma^*$, but rather an action on a language $L\subseteq\Sigma^\omega$. There is a natural extension of $L$ to $\widehat L\subset\Sigma^\Z$, which may be defined as follows. Consider the graph with vertex set $\{u^{-1}L:u\in\Sigma^*\}$, and put an edge from $u^{-1}L$ to $(u s)^{-1}L$ labeled $s$ for all $u\in\Sigma^*,s\in\Sigma$ such that $(u s)^{-1}L$ is nonempty. Then paths in this graph, starting at $L$, define elements of $L$, $\widehat L$ may be defined as the labels along bi-infinite paths in this graph that pass through $L$ at time $0$.

Consider a self-similar inverse semigroup family $(G_{ij})$ acting on $L$. Asymptotic equivalence is defined on $\widehat L$, as follows: given $(x_n)_{n\in\Z}$ and $(y_n)_{n\in\Z}$, they are declared to be asymptotically equivalent if there is a sequence of elements $(g_n)_{n\in\Z}$ in $\bigcup_{i,j}G_{i,j}$ with $\forall n:g_n(x_n x_{n+1}\dots)=y_n y_{n+1}\dots$ and $\#\{g_n\}<\infty$. If $x_n x_{n+1}\dots\in L_i$ and $y_n y_{n+1}\dots\in L_j$ then we require $g_n\in G_{i,j}$. As in the classical setting, we may in fact assume $g_n\in\mathscr N_{i,j}$.

The quotient of $\widehat L$ by asymptotic equivalence is the \emph{limit solenoid}. It is a \emph{Smale space} in the sense of Ruelle, Bowen et al. (see~\cite{ruelle:thermo,anderson-p:invariants}), and coincides in the case of substitutional subshifts with the  Smale-Williams solenoid~\cite{williams:expanding}. This will be detailed in~\cite{bartholdi-m:actions2}.

\section{Orbits}
Let $G$ be an inverse semigroup acting automatically on $L$, and consider the relations
\begin{align*}
  \mathscr O_G &=\{(x,x\cdot g)\mid x\in L,g\in G\}=\{(x,y)\mid x\cdot G=y\cdot G\},\\
  \mathscr P_G &=\{(x,y)\mid x,y\in L,\overline{x\cdot G}=\overline{y\cdot G}\},\\
  \mathscr Q_G &=\{(x,y)\mid x,y\in L,x\cdot G\subseteq\overline{y\cdot G}\},\\
  \mathscr R_G &=\{(x,y)\mid x,y\in L,x\in\overline{y\cdot G}\},\\
  \mathscr S_G &=\{(x,y)\mid x,y\in L,x\cdot G\cap\overline{y\cdot G}\neq\emptyset\},\\
  \mathscr T_G &=\{(x,y)\mid x,y\in L,\overline{x\cdot G}\cap\overline{y\cdot G}\neq\emptyset\}.
\end{align*}
Clearly $\mathscr O_G\subseteq \mathscr P_G\subseteq \mathscr Q_G\subseteq \mathscr R_G\subseteq \mathscr S_G\subseteq \mathscr T_G$, and $\mathscr O_G,\mathscr P_G$ are equivalence relations. If the $G$-action is not continuous, then none of these relations has to be closed.

There is in principle no reason for the relations $\mathscr O_G$, etc.\ to be $\omega$-automatic, and in fact there are cases in which they are not; see~\cite{bartholdi:wordorder} in which it is shown that $\{n\in\N\mid (0^\omega,1^n0^\omega)\in\mathscr O_G\}$ can be non-recursive, even for a contracting group $G$.

The situation is quite special for bounded automata: our main result is the
\begin{thm}\label{thm:orbits}
  Let $G$ be a finitely generated inverse semigroup acting boundedly and automatically on a language $L$. Then the relation $\mathscr O_G$ is effectively $\omega$-automatic.

  If furthermore the action of $G$ is topologically automatic, then the relations $\mathscr P_G,\mathscr Q_G,\mathscr R_G,\mathscr S_G,\mathscr T_G$ are also effectively $\omega$-automatic.
\end{thm}
\begin{proof}
  Let $G$ be an inverse semigroup generated by a collection $\mathscr A_1,\dots,\mathscr A_t$ of univalent, injective, bounded automata, let $\Sigma$ be the alphabet of $L$, and let $L_0,\dots,L_n$ denote the collection of all languages $u^{-1}L$ for $u\in\Sigma^*$. Let $(G_{i,j})_{0\le i,j\le n}$ denote the inverse semigroup family corresponding to $G$. By Proposition~\ref{prop:boundednuk}, there is a finite inverse semigroup family $(H_{i,j})_{0\le i,j\le n}$ and a finite collection of automata with a single cycle one of whose states is initial, that comprise the nucleus $(\mathscr N_{i,j})$ of $(G_{i,j})$. Let $\Xi$ denote the union of all states of $\mathscr A_1,\dots,\mathscr A_t$, of all $H_{i,j}$, and of all states on cycles of elements of all $\mathscr N_{i,j}$, all viewed as partial bijections $L_i\to L_j$ for some $i,j$. Note that $\Xi$ is a finite set.

  We construct a graph for an automaton recognizing $\mathscr O_G$, in two stages. To avoid confusion between states in $\Xi$ and states of the automaton recognizing $\mathscr O_G$, we call the latter \emph{nodes}.

  In the first stage, every node $q$ will carry a finite graph $\Gamma_q$, with edges labeled by $\Xi$ and two distinguished vertices respectively called the ``input'' and ``output''. The node $q$ represents the set $G_q\subset\Xi^*$ consisting of labels read along paths in $\Gamma_q$ from the input vertex to the output vertex, and also the relation $\mathscr O_q\coloneqq\{(x,x\cdot g)\mid x\in L,g\in G_q\}$. Implicitly, every edge of $\Gamma_q$ comes with two orientations, and the reverse orientation is labeled by the inverse of the automaton.

  The initial node $q_0$ carries a finite graph consisting of a single vertex and one loop for each initial state of some $\mathscr A_1,\dots,\mathscr A_t$, labeled by that state. We of course have $\mathscr O_G=\mathscr O_{q_0}$.

  To define the transitions from a node $q$, consider its graph $\Gamma_q$, with input and output vertices $i,o$. Form then the graph $\Gamma'$ with vertex set $\Gamma_q\times\Sigma$, and for all $a,b\in\Sigma$ give $\Gamma'$ an edge $(v,a)\to(w,b)$ labeled $s@(a,b)$ whenever $v\to w$ is an edge in $\Gamma_q$ labeled $s$. For all $a,b\in\Sigma$, we put a transition $q\to q'$ with label $(a,b)$, where the node $q'$ carries the graph $\Gamma'$ with input vertex $(i,a)$ and output vertex $(o,b)$.

  Whenever some node $q$ has the property that the input and output vertices either coincide or are connected by an element of the nucleus, we make this node residual.

  Note that the automaton we have constructed, up to now, is infinite; more precisely, it is a tree with $\Sigma^2$ descendants at every node.
  
  We now minimize this automaton. Firstly, we delete all nodes in which the input and output vertices are in different connected components, since they accept the empty relation. In each graph, we only keep the connected component containing both the input and output vertices. This does not affect the accepted relation.

  We then note that, because all transformations in $\Xi$ are bounded, there is a bound on the number of non-identity labels in the graphs $\Gamma_q$ associated with nodes. Indeed if $q$ is at distance $n$ from the root $q_0$ then all labels in $\Gamma_q$ are of the form $s@(u,v)$ for some $u,v\in\Sigma^n$ and some $s\in\Xi$, and every such label appears at most once for given $s,u,v$.

  We introduce ``shortening'' edges in the graphs $\Gamma_q$: whenever there are two consecutive $u\to v$ and $v\to w$ in some graph $\Gamma_q$, with respective labels $g,h\in\Xi$ such that $g h\in\Xi$, we add an edge $u\to w$ with label $g h$. Note that the number of non-trivial edges in each $\Gamma_q$ is still bounded.
  
  We next ``compress'' together all edges with finitary labels. Given any two vertices $u,v$ in some $\Gamma_q$, consider the set of all paths from $u$ to $v$ with a finitary as label. Since the inverse semigroup family $H$ generated by all finitaries is finite, and the set of all labels $F$ along such paths is a subset of $H$, we obtain that $F$ is finite. Let us add to $\Gamma_q$ an edge from $u$ to $v$, labeled $f$, for all $f\in F$. Then paths read in $\Gamma_q$ may be assumed never to contain two consecutive labels in $H$; therefore, every vertex of $\Gamma_q$ should have an edge with a non-finitary label on it, and we can remove from $\Gamma_q$ all vertices all of whose incident edges are finitary (except perhaps the input and output vertices). Since the number of edges in $\Gamma_q$ with a non-identity label is bounded by $B$, the number of vertices of $\Gamma_q$ is also bounded, by $2B+2$ counting both edge's extremity and the input and output vertices. There are therefore a finite number of possibilities for the graph $\Gamma_q$, We identify two nodes if their respective graphs are isomorphic (by an isomorphism preserving labels and input and output vertex).

  We have in this manner produced a finite automaton. Let us check that it accepts $\mathscr O_G$.
  
  On the one hand, consider $x=x_0x_1\dots$ and $y = x\cdot g = y_0y_1\dots$ for some $g\in G$; then there is a product $s_1\cdots s_\ell$ of initial states of various $\mathscr A_i$'s such that $x\cdot s_1\cdots s_\ell=y$, and therefore there is a path labeled $s_1\dots s_n$ in $\Gamma_{q_0}$. There is then a transition $q_0\to q_1$ labeled $(x_0,y_0)$ and a path in $\Gamma_{q_1}$ labeled $(s_1\dots s_\ell)@(x_0,y_0)$ from its input to its output; etc., eventually leading to a path in some $\Gamma_{q_n}$ labeled by $(s_1\dots s_\ell)@(x_0\dots x_n,y_0\dots y_n)$ with all $s_i$ in the nucleus $\mathscr N$. Recall that there is a constant $m$ such that $(s t)@(u,v)$ equals an element of the nucleus for all composable $s,t\in\mathscr N$ and all $u,v\in\Sigma^m$. Thanks to the ``shortening'' edges, we may replace consecutive pairs of letters in $(s_1\dots s_\ell)@(x_0\dots x_{n+m},y_0\dots y_{n+m})$ by single elements of the nucleus, obtaining a path in $\Gamma_{q_{n+m}}$ of length $\lceil\ell/2\rceil$; and repeating this $\lceil\log_2(\ell)\rceil$ times we obtain a path of length $1$ from the input to the output. We have arrived at residual states, so the pair $(x,y)$ is accepted by the automaton.

  Conversely, consider an accepting run for a pair of $\omega$-words $(x,y)$. For $n$ large enough, it will be at a state $q_n$ accepting the $(x_nx_{n+1}\dots,y_ny_{n+1}\dots)$ along a path of length $1$, labeled by a nucleus element. Working backwards, this path may be lifted, by inverting the ``shortening edges'' and ``compressing finitaries'' process, to a path in $\Gamma_{q_{n-1}}$ realizing the transformation $x_{n-1}x_n\dots\mapsto y_{n-1}y_n\dots$, etc., finally arriving at a path in $\Gamma_{q_0}$ which realizes the transformation $x\mapsto y$. By reading the elements of $\Xi$ along this path, we exhibit an element $g\in G$ with $x\cdot g=y$.

  Assume finally that the encoding $L\to X$ is topological, so there are neighbourhoods $(\mathcal U_n)_{n\in\N}$ of $\approx$ in $L\times L$ expressing the topology of $X$. The relations $\mathscr P_G,\mathscr Q_G,\mathscr R_G,\mathscr S_G,\mathscr T_G$ are defined by first-order statements, and are therefore also effectively $\omega$-automatic:
  \begin{align*}
    \mathscr R_G &= \{(x,y)\mid \forall n:\exists z:(x,z)\in\mathcal U_n\wedge (z,y)\in\mathscr O_G\},\\
    \mathscr P_G &= \mathscr R_G\cap\mathscr R_G^{-1},\\
    \mathscr Q_G &= \{(x,y)\mid \forall z:(x,z)\in\mathscr O_G\rightarrow (z,y)\in\mathscr R_G\},\\
    \mathscr S_G &= \{(x,y)\mid \exists z:(x,z)\in\mathscr O_G\wedge (z,y)\in\mathscr R_G\},\\
    \mathscr T_G &= \{(x,y)\mid \exists z:(x,z)\in\mathscr R_G\wedge (z,y)\in\mathscr R_G\}.\qedhere
  \end{align*}
\end{proof}

\begin{cor}\label{cor:infiniteorbit}
  It is decidable, given $G$ an inverse semigroup generated by finitely many bounded automata, whether $G$ has an infinite orbit.
\end{cor}
\begin{proof}
  By Theorem~\ref{thm:orbits}, the structure $(L,\mathscr O_G)$ is $\omega$-automatic. By Kuske-Lohrey~\cite{kuske-lohrey:counting}, its first order theory is decidable, even when enriched with the quantifier $\exists^\omega$ (``there exists countably many''). Thus the sentence ``$\exists x:\exists^\omega y: \mathscr O_G(x,y)$'' is decidable.
\end{proof}

\begin{cor}\label{cor:infinite}
  It is decidable, given $G$ an inverse semigroup generated by finitely many bounded automata, whether $G$ is infinite.
\end{cor}
\begin{proof}
  By D'Angeli-Francoeur-Rodaro-W\"achter~\cite{dangeli-francoeur-rodaro-wachter:orbits}, a semigroup of univalent automata is infinite if and only if it has an infinite orbit.
\end{proof}

\begin{cor}
  It is decidable, given an injective, univalent bounded automaton $\mathscr A$, whether $\mathscr A$ has finite or infinite order.\qed
\end{cor}

We note that~\cite{bartholdi:domino} already implies decidability of the orbit relation: indeed there is it shown that the monadic second-order logic of a bounded action (in which quantification is over points of the space, and the structure encodes the action of group generators $S$) is decidable. The fact that $x,y\in X$ are in the same $G$-orbit is expressed as the monadic second-order predicate
\[\neg\exists A,B:x\in A\wedge y\in B\wedge A\cap B=\emptyset\wedge\bigwedge_{s\in S} A\cdot s=A\wedge B\cdot s=B.\]
However, the decision procedure involves tree automata rather than Büchi automata.

\section{Examples and illustrations}
We give here some examples of orbit calculations, following Theorem~\ref{thm:orbits}. Applications to complex dynamics will be given in the next section.

\subsection{The Bernoulli shift}\label{ss:bernoulli}

We give a natural example of automatic action of $\Z$ whose orbit relation is \emph{not} automatic: the Bernoulli shift. As a warm-up, the \emph{one-sided} Bernoulli shift on an alphabet $\mathcal A$ with at least two symbols $\zero,\one$ is the space $\mathcal A^\N$, with shift map $S\colon\mathcal A^\N\righttoleftarrow$ given by $S(a_0a_1\dots)=a_1a_2\dots$. The action of $S$ is automatic, and may be given by the following automaton: it has stateset $\mathcal A$, all states are initial and residual, and for all $a,b\in\mathcal A$ there is a transition from $a$ to $b$ labeled $(a,b)$. In effect, when the automaton is in state $a$, it ``checks'' that the input is $a$, ``guesses'' the next letter $b$, and moves to state $b$ so as, later, to check that its guess is correct. 

We consider now the (two-sided) Bernoulli shift. It is the space $\mathcal A^\Z$, with action of $\Z$ given by shifting.
\begin{prop}
  The Bernoulli shift admits an automatic presentation, for which the orbit relation is not automatic.
\end{prop}
\begin{proof}
  It is straightforward to encode $\mathcal A^\Z$ into an $\omega$-regular language, by the ``conveyor belt'' trick: set $\Sigma=\mathcal A\times\mathcal A$, and encode $(a_n)_{n\in\Z}$ by the sequence $(a_0,a_{-1})(a_1,a_{-2})\dots$. Thus the map $\Sigma^\omega\to\mathcal A^\Z$ is a bijection.

  The action of $\Z$ on $\mathcal A^\Z$ is generated by the shift $S\colon\mathcal A^\Z\righttoleftarrow$ given by $S((a_n)) = (a_{n+1})_{n\in\Z}$. As a transformation of $\Sigma^\omega$, it is given by the following automaton: it has $\Sigma$ as stateset, all states $(a,a)$ are initial, all states are residual, and there is a transition from $(a,b)$ to $(c,d)$ labeled $((a,d),(c,b))$, for all $a,b,c,d\in\mathcal A$. The run accepting $(a_0,a_{-1})(a_1,a_{-2})\dots\mapsto(a_1,a_0)(a_2,a_{-1})\dots$ passes through states $(a_0,a_0)$, $(a_1,a_{-1})$, $(a_2,a_{-2})$, \dots.

  Assume now for contradiction that the orbit relation $\mathscr O_\Z$ is automatic. For all $n\in\N$, the pair $(\dots\zero\underline{\one^n}\zero\dots,\dots\zero\underline{\zero^n}\one^n\zero\dots)$ is accepted, since these two sequences are in the same $\Z$-orbit. A successful run, for such pairs, must start at an initial state, read repeatedly $a=((\one,\zero),(\zero,\zero))$, then read repeatedly $b=((\zero,\zero),(\one,\zero))$, and finally read infinitely many $c=((\zero,\zero),(\zero,\zero))$. Now after repeatedly reading the $a$, it will enter a loop, namely there is a state $q$ that is reached by reading $a^m$, and also $a^n$ for some $n>m$. Then, from $q$ and reading $b^m$, it will reach a state from which reading $c^\omega$ leads to a residual cycle. This run also accepts $a^n b^m c^\omega$, which is not in $\mathscr O_\Z$, and we have reached our desired contradiction.
\end{proof}

\subsection{The Grigorchuk group}
Consider, on the other hand, the Grigorchuk group $G$ that appeared in~\eqref{eq:grigorchuk}. Its orbit relation is tail equivalence; namely, two sequences are equivalent if and only if they differ in finitely many positions. An automaton accepting $\mathscr O_G$ is
\[\begin{fsa}
    \node[state,initial] (a) at (0,0) {};
    \node[state,accepting by double] (b) at (3,0) {};
    \draw (a) edge[loop below] node[below] {$\{0,1\}\times\{0,1\}$} ()
    edge node[above] {$\{0,1\}\times\{0,1\}$} (b)
    (b) edge[loop right] node[right] {$\Delta0,\Delta1.$} ();
  \end{fsa}
\]  

For a more complicated example, consider the subgroup $H=\langle abab,d,ada\rangle$ of the Grigorchuk group. An automaton accepting $\mathscr O_H$ is
\[\begin{fsa}
    \node[state,initial] (q) at (0,0) {$q$};
    \node[state] (r) at (2.5,0) {$r$};
    \node[state] (s00) at (5,2.25) {$s_{00}$};
    \node[state] (s01) at (5,0.75) {$s_{01}$};
    \node[state] (s10) at (5,-0.75) {$s_{10}$};
    \node[state] (s11) at (5,-2.25) {$s_{11}$};
    \node[state] (t) at (7.5,0) {$t$};
    \node[state,accepting by double] (u) at (10,0) {$u$};
    \node[state,accepting by double] (vb) at (6.5,-2) {$v_b$};
    \node[state,accepting by double] (vc) at (7.5,-2) {$v_c$};
    \node[state,accepting by double] (vd) at (8.5,-2) {$v_d$};
    \draw (q) edge node[above] {$\Delta0,\Delta1$} (r)
    (r) edge node[above,sloped] {$\Delta0$} (s00)
    (r) edge node[above,sloped] {$(0,1)$} (s01)
    (r) edge node[above,sloped] {$(1,0)$} (s10)
    (r) edge node[above,sloped] {$\Delta1$} (s11)
    (s00) edge node[above,sloped] {\tiny $\{0,1\}\times\{0,1\}$} (t)
    (s01) edge node[above,sloped] {\tiny $\{0,1\}\times\{0,1\}$} (t)
    (s10) edge node[above,sloped] {\tiny $\{0,1\}\times\{0,1\}$} (t)
    (s11) edge node[above,sloped] {\tiny $\{0,1\}\times\{0,1\}$} (t)
    (t) edge[loop above] node[above] {\tiny $\{0,1\}\times\{0,1\}$} ()
    edge node[sloped] {$\Delta1$} (vb)
    edge node[sloped] {$\Delta1$} (vc)
    edge node[sloped] {$\Delta1$} (vd)
    edge node[above] {$\Delta0,\Delta1$} (u)
    (u) edge[loop right] node {$\Delta0,\Delta1$} ()
    (vb) edge[loop below] node[below] {$\Delta1$} ()
    (vc) edge[loop below] node[below] {$\Delta1$} ()
    (vd) edge[loop below] node[below] {$\Delta1$} ();
  \end{fsa}
\]  
The graphs $\Gamma_q,\dots$ represented at the nodes are as follows:
\[\Gamma_q=\begin{fsa}[baseline=0]
    \node[state,initial,accepting] (a) at (0,0) {};
    \draw (a) edge[loop above] node {$abab$} ()
    (a) edge[out=-45,in=-75,loop] node[below] {$d$} (a)
    (a) edge[out=-105,in=-135,loop] node[below] {$ada$} (a);
  \end{fsa};\qquad
  \Gamma_r = \begin{fsa}[baseline=0]
    \node[state,initial,accepting] (a) at (0,0) {};
    \draw (a) edge[loop above] node {$ca$} ()
    (a) edge[loop below] node[below] {$b$} ();
  \end{fsa};\qquad
  \Gamma_{s_{01}} = \begin{fsa}[baseline=0]
    \node[state,initial] (a0) at (0,1) {};
    \node[state,accepting] (a1) at (0,-1) {};
    \draw (a0) edge[loop above] node {$a$} ()
    (a0) edge[bend left=30] node[right] {$a$} (a1)
    (a1) edge[loop below] node[below] {$c$} ()
    edge [bend left=30] node[left] {$d$} (a0);
  \end{fsa}
\]
and similarly for $s_{00},s_{10},s_{11}$ with different input and output states; and
\[\Gamma_t=\begin{fsa}[baseline=0]
    \node[state,initial,accepting] (a) at (0,0) {};
    \draw (a) edge[loop above] node {$a$} ()
    (a) edge[out=-45,in=-75,loop] node[below] {$d$} (a)
    (a) edge[out=-105,in=-135,loop] node[below] {$b$} (a);
  \end{fsa};\qquad
  \Gamma_u = \begin{fsa}[baseline=0]
    \node[state,initial,accepting] (a) at (0,0) {};
  \end{fsa};\qquad
  \Gamma_{v_b} = \begin{fsa}[baseline=0]
    \node[state,initial,accepting] (a) at (0,0) {};
    \draw (a) edge[loop above] node {$b$} ();
  \end{fsa}
\]
and similarly for $v_c,v_d$.

\subsection{The adding machine}
Recall the automaton adding $1$ in base $d$, from~\eqref{eq:adder}; it is the iterated monodromy group of the polynomial $z^d$. The orbits of $\Z$, acting on sequences in $\{\ell_{0/d},\dots,\ell_{(d-1)/d}\}^\omega$, are cofinality classes, except that the classes cofinal to $\ell_{0/d}^\omega$ and $\ell_{(d-1)/d}^\omega$ are in the same orbit. Thus an automaton for the orbit relation is
\[\begin{fsa}
    \node[state,initial] (q) at (0,0) {};
    \node[state,accepting by double] (r) at (3,0) {};
    \node[state,accepting by double] (s+) at (-1,-1) {};
    \node[state,accepting by double] (s-) at (1,-1) {};
    \draw (q) edge[loop above] node[above] {$\{(\ell_{i/d},\ell_{j/d})\}$} ()
    edge node[left] {$\{(\ell_{i/d},\ell_{j/d})\}$} (s+)
    edge node[right] {$\{(\ell_{i/d},\ell_{j/d})\}$} (s-)
    edge node[above] {$\{(\ell_{i/d},\ell_{j/d})\}$} (r)
    (r) edge[loop right] node[right] {$(\ell_{i/d},\ell_{i/d})\}$} ()
    (s+) edge[loop left] node[left] {$(\ell_{0/d},\ell_{(d-1)/d})\}$} ()
    (s-) edge[loop right] node[right] {$(\ell_{(d-1)/d},\ell_{0/d})\}$} ();
  \end{fsa}
\]

\subsection{The nega-Zeckendorf and Fibonacci subshifts}
We begin with the nega-Zeckendorf representation, and the automaton~\eqref{eq:negazeckendorf}. It has the advantage that its states know the next input and output symbols; thus the states representing $t$ are actually $t_{10},t_{01},t_{00}$, and the identities are actually $1_{00},1_{11}$. Starting from the initial state represented by the graph $\Gamma_0$ with a single vertex and loops $t_{00},t_{01},t_{10}$, we follow the automaton~\eqref{eq:negazeckendorf} to obtain the successors of $q_0$: at the next step, the graph is
\[\Gamma_1:\begin{fsa}[baseline=0]
    \node[state] (0) at (0,0) {$0$};
    \node[state] (1) at (3,0) {$1$};
    \draw (0) edge[loop left] node {$u_{10}$} ()
    (0) edge[bend left=15] node {$1_{00}$} (1)
    (1) edge[bend left=15] node {$u_{00},u_{01}$} (0);
  \end{fsa}\]
with choices of input and output states determined by the transition from $q_0$. One step further, we see
\[\Gamma_2:\begin{fsa}[baseline=0]
    \node[state] (00) at (0,1) {$00$};
    \node[state] (01) at (0,-1) {$01$};
    \node[state] (10) at (3,0) {$10$};
    \draw (01) edge[bend left] node {$1_{00}$} (00)
    (00) edge[bend left=15] node[sloped] {$1_{00},1_{11}$} (10)
    (10) edge[bend left=15] node[sloped,below] {$t_{01}$} (00)
    (10) edge[bend left=15] node[sloped,below] {$t_{00},t_{10}$} (01);
  \end{fsa}\]
and now the edge labeled $1_{00},1_{11}$ carries the identity so may be contracted, leading to the previous graph $\Gamma_1$, except that the edges are labeled by its inverses. Furthermore, the states $10$ and $01$ are connected by an element on a cycle; so the automaton accepting $\mathscr O_{\langle t\rangle}$ should transition either to a residual state recognizing the identity, or to a residual state recognizing the cycle of $t_{10}$. The automaton we arrive at, after merging $\Gamma_0$ and $\Gamma_1$ and distinguishing all $\Gamma_1$'s by their input and output state, is
\[\begin{fsa}
    \node[state,initial] (11) at (3,3) {$11$};
    \node[state,initial] (00) at (3,1) {$00$};
    \node[state,initial left] (01) at (2,-1) {$01$};
    \node[state,initial right] (10) at (4,-1) {$10$};
    \node[state,accepting by double] (i0) at (6,2) {};
    \node[state,accepting by double] (i1) at (6,0) {};
    \node[state,accepting by double] (x01) at (4,-3) {};
    \node[state,accepting by double] (x10) at (2,-3) {};
    \draw (11) edge[bend left=15] node[sloped] {$\Delta1$} (00)
          (11) edge node[sloped] {$\Delta1$} (i0)
          (00) edge[loop left] node[left] {$\Delta0$} (00)
          (00) edge[bend left=15] node[sloped] {$\Delta0$} (11)
          (00) edge[bend left=15] node[sloped] {$\Delta0$} (01)
          (00) edge[bend left=15] node[sloped] {$\Delta0$} (10)
          (00) edge node[sloped] {$\Delta0$} (i0)
          (00) edge node[sloped] {$\Delta0$} (i1)
          (10) edge[bend left=15] node[sloped] {$(1,0)$} (00)
          (10) edge[bend left=15] node[sloped] {$(1,0)$} (01)
          (10) edge node[sloped] {$(1,0)$} (x01)
          (01) edge[bend left=15] node[sloped] {$(0,1)$} (00)
          (01) edge[bend left=15] node[sloped] {$(0,1)$} (10)
          (01) edge node[sloped] {$(0,1)$} (x10);
    \draw (i0) edge[loop above] node {$\Delta0$} ()
          (i0) edge[bend left=15] node {$\Delta0$} (i1)
          (i1) edge[bend left=15] node {$\Delta1$} (i0);
    \draw (x01) edge[bend left=15] node {$(0,1)$} (x10)
          (x10) edge[bend left=15] node[above] {$(1,0)$} (x01);
  \end{fsa}
\]

The orbit automaton for the Fibonacci action is quite similar; it also recognizes cofinal sequences, and identifies the sequences $(13)^\omega$ with $(21)^\omega$ and $(31)^\omega$ with $(12)^\omega$. It is just too big to fit nicely on a page (having non-residual states $\{1,2,3\}\times\{1,2,3\}$ and seven residual states).

\section{Complex dynamics}\label{ss:cx}
\let\cev=\overleftarrow 
The iterated monodromy groups presented in~\S\ref{ss:img} are part of a general theory developed by Nekrashevych (see~\cite{nekrashevych:ssg}), and establishing an equivalence between contracting self-similar groups and expanding self-maps. There is indeed a reverse construction, the \emph{limit space} of a contracting group, which we briefly summarize. Consider a contracting self-similar group $G$ acting on a rooted $\Sigma^*$. Let $\Sigma^{-\omega}$ denote the space of left-infinite words, and define the following equivalence relation $\sim$ on $\Sigma^{-\omega}$:
\begin{multline*}
  (x_n)_{n\le0}\sim(y_n)_{n\le0}\text{ if and only if }\exists(g_n)_{n\le0}\text{ in $G$ with }\\
  \#\{g_n:n\le0\}<\infty\text{ and }\forall n:g_n(x_n\dots x_0)=y_n\dots y_0.
\end{multline*}
Letting $\mathscr N$ denote the nucleus of $G$, this is equivalent to ``$\exists (g_n)\text{ in }\mathscr N:\forall n:g_n(x_n\dots x_0)=y_n\dots y_0$''. Therefore, $\sim$ is a closed equivalence relation.

Restricting ourselves to automata acting on a rooted tree $\Sigma^*$ rather than on $\Sigma^\omega$ is a convenience; all we need is the action to be by homeomorphisms, so that the automata representing them accept closed languages, and therefore may be assumed to have all states residual.

The quotient $\Sigma^{-\omega}/{\sim}$ is called the \emph{limit space} $\mathcal J(G)$ of $G$. It is a metrizable topological space of finite covering dimension, connected and locally connected as soon as $G$ is finitely generated. The one-sided shift map on $\Sigma^{-\omega}$ induces a $\#\Sigma$-to-$1$ expanding map $f$ on $\mathcal J(G)$, and (under an additional assumption about germs of the action of $G$) the iterated monodromy group of $(\mathcal J(G),f)$ recovers $G$ itself. This holds, for example, for the example of the Basilica polynomial and group given in~\S\ref{ss:img}. Indeed, the limit space $\mathcal J(G)$, when $G$ is the iterated monodromy group of a post-critically finite complex rational map $f$, is homeomorphic to the \emph{Julia set} of $f$, namely the set of points $z\in\widehat\C$ at which the forward orbit $(z,f(z),f(f(z)),\dots)$ does not vary continuously.

In the parlance of~\S\ref{ss:automatic}, the limit space $\mathcal J(G)$, with its self-map $f$, admits an automatic presentation. Let us indeed define, for an automaton $\mathscr A$, its reverse $\cev{\mathscr A}$ as the automaton with same stateset as $\mathscr A$ but with arrows in reverse direction. (We ignore the initial and residual states). Then the equivalence relation $\sim$ defined above is coded by $\cev{\mathscr N}$, the reverse of the nucleus of $G$. It follows that $\mathscr J(G)$ is presented as $\Sigma^\omega/\cev{\mathscr N}$, with the usual topology on $\Sigma^\omega$ (which is automatic, see the remarks at the end of~\S\ref{ss:automatic}), and the shift map on $\Sigma^\omega$, which is also given by an automaton.

It is possible to extend the theory of limit spaces to the setting of a contracting self-similar inverse semigroup $G$ acting an $\omega$-regular language $L$, when $L$ is closed; see~\cite{nekrashevych:smale} for its geometric development. In the language of automata, one considers the collection $L_0=L,\dots,L_n$ of languages of the form $u^{-1}L$ as the stateset of an automaton, with an edge labeled $\sigma$ from $L_i$ to $L_j$ whenever $L_j=\sigma^{-1}L_i$; then $L$ is naturally identified with the space of right-infinite paths that start at $L_0$ in this graph.

Let us consider in more detail the situation of a post-critically finite complex rational map $f$ and its iterated monodromy group $G$, which is a contracting self-similar group. If $f$ is a polynomial, then there is a choice of spider $\Sigma$ such that $G$ is bounded; and conversely, if $G$ is bounded then then the Julia set of $f$ has \emph{local cut points}, namely there is a dense set of points whose complement is not locally connected. (It is often difficult to find a good spider, but polynomials always admit a good one, obtained by choosing the basepoint and the spider inside the Hubbard tree).

If $G$ is a bounded self-similar group, then it then follows from Theorem~\ref{thm:orbits} that the orbit relation $\mathscr O_G$ is automatic, and moreover that this holds for all finitely generated subgroups of $G$. A simple, yet interesting example of such subgroup $\langle g_z\rangle$, where $g_z\in G$ is an element that surrounds once a periodic critical point $z\in\widehat\C$; this point is the center of a \emph{Fatou component}, and the Julia set contains a topological circle which is the boundary of this Fatou component.

Now, we can always put ourselves in a situation in which $\langle g_z\rangle$ is a bounded, contracting group. Up to replacing $f$ by $f^m$ for $m=$ the period of $z$, we may assume that $z$ is a fixed point for $f$; this does not change the Julia set nor the group $G$, though it will now act on the isomorphic language $(\Sigma^m)^\omega$. Choose the basepoint $*$ very close to $z$, so $f^{-m}(*)$ contains $d$ points $*_1,\dots,*_d$ very close to $z$, where $d$ is the local degree of $f^m$ at $z$.

We may now perform a \emph{change of spider}: in the spider $\Sigma^m$, there will be $d$ elements $\ell_1,\dots,\ell_d$ that connect $*$ to respectively $*_1,\dots,*_d$. For some elements $g_1,\dots,g_d\in G$, their representatives $\tilde g_1,\dots,\tilde g_d$ in $\pi_1(\mathscr M,*)$ will have the property that all $\ell_i'\coloneqq\ell_i\#\tilde g_i$ are paths that remain close to $z$, more precisely that remain in the Fatou component of $z$. Furthermore, for other spider legs $\ell_j$ that connect $*$ to a preimage $*_j$ very close to a preperiodic preimage of $z$, modify them as $\ell_j'\coloneqq\ell_j\#\tilde g_j$ for some $\tilde g_j$ in such a manner that all spider legs approaching a given preimage of $z$ remain close to each other. Let $\Sigma'$ be the spider for $f^m$ in which the paths $\ell_i$ have all been replaced by $\ell_i',\dots,\ell_d'$.

On the one hand, this guarantees that, if $g_z$ is now a loop based at $*$ that surrounds once $z$ and remains in $z$'s Fatou component, then the action of $g_z$ is bounded. Indeed the automaton representing the action on $\{\ell_1',\dots,\ell_d'\}^\omega$ is precisely given by~\eqref{eq:adder}; while the action on other sequences is finitary, because the lifts of $z$ distinct from $z$ itself are preperiodic and we chose the spider legs so that their action is finitary.

On the other hand, the group $G$ also acts on $(\Sigma')^\omega$ as a self-similar contracting group, and there is an automatic homeomorphism $h\colon(\Sigma')^\omega\to(\Sigma^m)^\omega$ that conjugates one action to the other~\cite{nekrashevych:ssg}*{Corollary~2.11.7}. We have proven:
\begin{prop}
  Let $g_z$ be any element of $G$ that may be represented by a loop surrounding a single periodic critical point; then the orbit relation $\mathscr O_{\langle g_z\rangle}$ is automatic.\qed
\end{prop}

If $\mathscr O_H$ is an automaton representing the orbit relation for a group $H$, then $\cev{\mathscr O}$ is an equivalence relation: indeed $(x,y)\in\cev{\mathscr O_H}$ if and only if there exists $(h_n)_{n\le0}$ in $H$ with $h_n(x_n\dots x_0)=y_n\dots y_0$, and this last condition clearly defines an equivalence relation. In the case of $H=\langle g_z\rangle$, this equivalence relation $\cev{\mathscr O}$ says: ``$x,y$ are on the boundary of a preimage of the Fatou component of $z$''. Furthermore, there is a ``main action'' of $H$ on $\{\ell_1',\dots,\ell_d'\}^\omega$ that is encoded by an automatic relation, and defines the relation ``$x,y$ are on the boundary of the Fatou component of $z$''. Even better, every Fatou component that is a preimage of the Fatou component of $z$ is identified by an address in $(\Sigma'\setminus\{\ell'_1,\dots,\ell'_d\})(\Sigma')^*$, by tracking which iterated preimage of $*$ it contains. We have obtained:
\begin{cor}\label{cor:fatou}
  For every periodic critical point $z$, the relation ``$x,y$ are on the boundary of a preimage of the Fatou component of $z$'' is effectively automatic. For every Fatou component $F$, the relation ``$x,y$ are on the boundary of $F$'' is effectively automatic.\qed
\end{cor}

Let us illustrate this with some computations for the Basilica map $f(z)=z^2-1$ and its iterated monodromy group $G$; the automaton was given in~\eqref{eq:basilica}. The orbit relation $\mathscr O_G$ is given by
\[\begin{fsa}
    \node[rectangle,draw,initial] (q) at (0,0) {\begin{fsa}\node[state,initial,accepting] (qq) at (0,0) {}; \draw (qq) edge[loop above] node {$a$} () edge[loop below] node {$b$} ();\end{fsa}};
    \node[rectangle,draw,double] (x) at (5,0) {\begin{fsa}\node[state,initial,accepting] at (0,0) {};\end{fsa}};
    \node[rectangle,draw,double] (a) at (-4,-3) {\begin{fsa}\node[state,initial,accepting] (qa) at (0,0) {}; \draw (qq) edge[loop above] node {$a$} ();\end{fsa}};
    \node[rectangle,draw,double] (b) at (-1,-3) {\begin{fsa}\node[state,initial,accepting] (qb) at (0,0) {};\draw (qq) edge[loop below] node {$b$} ();\end{fsa}};
    \node[rectangle,draw,double] (A) at (2,-3) {\begin{fsa}\node[state,initial,accepting] (qa) at (0,0) {}; \draw (qq) edge[loop above] node {$a$} ();\end{fsa}};
    \node[rectangle,draw,double] (B) at (5,-3) {\begin{fsa}\node[state,initial,accepting] (qb) at (0,0) {};\draw (qq) edge[loop below] node {$b$} ();\end{fsa}};
    
    \draw (q) edge[loop above,looseness=4] node {$\Sigma\times\Sigma$} (q)
          (q) edge node[sloped] {$\Sigma\times\Sigma$} (a.north)
          (q) edge node[sloped] {$\Sigma\times\Sigma$} (b)
          (q) edge node[sloped] {$\Sigma\times\Sigma$} (A)
          (q) edge node[sloped] {$\Sigma\times\Sigma$} (B.north)
          (q) edge node {$\Sigma\times\Sigma$} (x);
    \draw (a) edge[bend left] node {$(\ell_{-*},\ell_*)$} (b);
    \draw (b) edge[bend left] node {$(\ell_{-*},\ell_{-*})$} (a);
    \draw (A) edge[bend left] node {$(\ell_*,\ell_{-*})$} (B);
    \draw (B) edge[bend left] node {$(\ell_{-*},\ell_{-*})$} (A);
    \draw (x) edge[loop above] node {$\Delta(\Sigma)$} ();
  \end{fsa}\]

Consider now the subgroup $H=\langle a\rangle$. The $H$-orbit relation is
\[\begin{fsa}
    \node[rectangle,draw,initial] (ia) at (0,0) {\begin{fsa}\node[state,initial,accepting] (qia) at (0,0) {}; \draw (qia) edge[loop above] node {$a$} ();\end{fsa}};
    \node[rectangle,draw] (ib) at (4,0) {\begin{fsa}\node[state,initial,accepting] (qib) at (0,0) {}; \draw (qib) edge[loop above] node {$a$} ();\end{fsa}};
    \node[rectangle,draw,double] (x) at (8,0) {\begin{fsa}\node[state,initial,accepting] at (0,0) {};\end{fsa}};
    \node[rectangle,draw,double] (a) at (-2,-3) {\begin{fsa}\node[state,initial,accepting] (qa) at (0,0) {}; \draw (qq) edge[loop above] node {$a$} ();\end{fsa}};
    \node[rectangle,draw,double] (b) at (1,-3) {\begin{fsa}\node[state,initial,accepting] (qb) at (0,0) {};\draw (qq) edge[loop below] node {$b$} ();\end{fsa}};
    \node[rectangle,draw,double] (A) at (4,-3) {\begin{fsa}\node[state,initial,accepting] (qa) at (0,0) {}; \draw (qq) edge[loop above] node {$a$} ();\end{fsa}};
    \node[rectangle,draw,double] (B) at (7,-3) {\begin{fsa}\node[state,initial,accepting] (qb) at (0,0) {};\draw (qq) edge[loop below] node {$b$} ();\end{fsa}};
    
    \draw (ia) edge[bend left] node {$\Sigma\times\Sigma$} (ib)
          (ia) edge node[sloped] {$\Sigma\times\Sigma$} (a.north)
          (ia) edge node[sloped] {$\Sigma\times\Sigma$} (b)
          (ia) edge node[sloped] {$\Sigma\times\Sigma$} (A)
          (ia) edge node[sloped] {$\Sigma\times\Sigma$} (B.north);
    \draw (ib) edge[bend left] node {$(\ell_*,\ell_*)$} (ia)
          (ib) edge node {$(\ell_{-*},\ell_{-*})$} (x);
    \draw (a) edge[bend left] node {$(\ell_{-*},\ell_*)$} (b);
    \draw (b) edge[bend left] node {$(\ell_{-*},\ell_{-*})$} (a);
    \draw (A) edge[bend left] node {$(\ell_*,\ell_{-*})$} (B);
    \draw (B) edge[bend left] node {$(\ell_{-*},\ell_{-*})$} (A);
    \draw (x) edge[loop above] node {$\Delta(\Sigma)$} ();
  \end{fsa}\]

To determine the relation ``belonging to the boundary of the same Fatou component'', we consider the closure of $\mathscr O_H$, and reverse the direction of the arrows, giving the relation
\[\begin{fsa}
    \node[state,initial] (q) at (0,0) {$q$};
    \node[state] (r) at (3,0) {$r$};
    \node[state] (s) at (6,0) {$s$};
    \draw (q) edge[loop above] node {$\Delta(\Sigma)$} ()
          (q) edge node {$(\ell_{-*},\ell_{-*})$} (r);
    \draw (r) edge[bend left] node {$\Sigma\times\Sigma$} (s);
    \draw (s) edge[bend left] node {$(\ell_*,\ell_*)$} (r);
  \end{fsa}\]
Quotienting the Julia set by the equivalence relation generated by this relation amounts to shrinking each Fatou component to a point; this operation produces an everywhere-expanding map. (In this case, the resulting space is a point; but, in the case of the ``airplane'' polnomial $f(z)\approx z^2-1.754878$, it yields a dendrite).

Considering the subrelation in which some prefix is read along the loop at $q$, followed by the edge from $q$ to $r$, defines the relation ``belonging to the boundary of a specific Fatou component'', determined by the given prefix.


\end{document}